\documentclass[11pt, reqno]{amsart}
\usepackage{cite}
\usepackage{wrapfig, lipsum,booktabs}
\usepackage{graphicx}
\usepackage{amssymb}
\usepackage{epstopdf}
\usepackage{verbatim}
\usepackage{bm}
\usepackage{multicol}
\usepackage{multirow}
\usepackage{subfigure}
\usepackage{fancyhdr} 
\usepackage{float}
\usepackage{stmaryrd}
\usepackage{color}
\usepackage{soul}
\usepackage{tikz}
\usepackage{algorithm,algorithmic}
\usepackage{todonotes}
\usetikzlibrary{patterns}
\usepackage{pgfplots}
\usepackage{cancel}
\newtheorem{theorem}{Theorem}[section]

\newtheorem{remark}{Remark}[section]

\usepackage{multirow}
\usepackage{amsmath,amssymb,eucal}
\usepackage{graphicx,subfigure,epsfig}
\usepackage{psfrag}
\usepackage{url}
\usepackage[top=1in, bottom=1.in, left=1in, right=1in]{geometry}

\newcommand{\bld}[1]{\hbox{\boldmath$#1$}}    
\newcommand{\Oh}{\Omega_h}
\newcommand{\pol}{\mathcal{P}}
\setulcolor{red}
\newcommand{\jmp}[1]{[\![#1 ]\!]}
\newcommand{\avg}[1]{\{#1 \}}

\begin{document}
\title[DG for Shallow Water]{A high-order velocity-based discontinuous Galerkin scheme for the shallow water equations: local conservation, entropy stability, well-balanced property, and positivity preservation}
\author{Guosheng Fu}
\address{Department of Applied and Computational Mathematics and 
Statistics, University of Notre Dame, USA.}
\email{gfu@nd.edu}
 \thanks{We acknowledge the partial support of this work
 from U.S. National Science Foundation through grant DMS-2012031.
 }

\begin{abstract}
  The nonlinear shallow water equations (SWEs) are widely used to model the
  unsteady water flows in rivers and coastal areas.
  In this work, we present  a novel class of 
locally conservative, entropy stable and well-balanced 
discontinuous Galerkin (DG) methods for the nonlinear shallow water equation with 
 a non-flat bottom  topography.
The major novelty of our work is the use of velocity field as an independent solution unknown in the DG scheme, which is closely related to the entropy variable approach to entropy stable schemes for system of conservation laws proposed by Tadmor \cite{Tadmor86} back in 1986, where recall that velocity is part of the entropy variable for the shallow water
equations. Due to the use of velocity as an independent solution unknown, no specific numerical quadrature rules are needed to achieve entropy stability of
our scheme on general unstructured meshes in two dimensions.

The proposed DG semi-discretization is then carefully combined with the classical
explicit
strong stability preserving Runge-Kutta (SSP-RK) time integrators \cite{GST01}
to yield a locally conservative, well-balanced, and positivity preserving fully discrete scheme. 
Here the positivity preservation property is enforced with the help of a simple scaling limiter.
In the fully discrete scheme, we re-introduce discharge as an auxiliary unknown
variable. In doing so, standard slope limiting procedures can be applied 
on the conservative variables (water height and discharge) without violating the
local conservation property. Here we apply a characteristic-wise TVB limiter \cite{CS98} 
on the conservative variables using the Fu-Shu troubled cell indicator \cite{FS17} in each inner
stage of the Runge-Kutta time stepping to suppress numerical oscillations.
This fully discrete can be readily applied to various SWEs simulations without dry areas where the water height is close to {\it zero}. 

The case with dry areas need further special attention, where the velocity approximation can be unphysically large near cells with a small water height, which may eventually crashes the simulation if no special treatment is used near these cells. Here we propose a simple wetting/drying treatment for the velocity update without violating the local conservation property to enhance the robustness of the overall scheme. 

One- and two-dimensional numerical experiments are presented to demonstrate the
performance of the proposed methods. 
\end{abstract}
\maketitle

\section{Introduction}
\label{sec:intro}
The system of nonlinear shallow water equations (SWEs) is a mathematical model for the fluid movement in
various shallow water environments, where the horizontal scales of motion are
much greater than the vertical scale. The SWEs have been widely 
used to model flow in the river, near-shore ocean, and  earth's atmosphere, etc.
In two dimensions, the inviscid SWEs take the following form:
\begin{subequations}
  \label{swe}
  \begin{alignat}{2}
    \label{swe1}
    h_t + \nabla\cdot (h\bld u)=&\;0,\\
    \label{swe2}
    (h\bld u)_t + \nabla\cdot (h\bld u\otimes \bld u)
    +\frac12g \nabla (h^2) =&\;-g h\nabla b,
  \end{alignat}
\end{subequations}
where $h$ is the water height, $\bld u = (u, v)$ is the velocity field, 
$b(x,y)$ represents the bottom topography and $g$ is the gravitational
constant. 

Below we review the four important properties that the SWEs satisfy, namely the
entropy condition, the {\it lake-at-rest} well-balanced property, 
the
positivity of the water height $h$, and the conservation property.
The system \eqref{swe} is a system of {\it balance laws},
\begin{align}
 \label{bal}
  U_t  + F(U)_x+G(U)_y = -s(x,y,U),
\end{align}
where $U=[h, hu, hv]^T$ is the vector of unknowns, 
$F=[hu, hu^2+\frac12 gh^2, huv]^T$ and $G=[hv, huv, hv^2+\frac12 gh^2]^T$ are
flux vectors, and $s=[0, ghb_x, ghb_y]^T$ is the source vector.
It is well-known that solutions of the balance law \eqref{bal} can develop shock
discontinuities in a finite time, independent of whether the initial data is
smooth or not. Hence, the solution of \eqref{swe} are considered in the
weak sense \cite{D16}, which are in general not unique. 

\textbf{(i) The entropy condition}.
To address the issue of non-uniqueness for weak solutions,
an additional admissibility criterion based on the so-called {\it entropy
condition} is imposed.
For the SWEs, the total energy 
\[
  E(U) :=\frac12h(u^2+v^2)+\frac12gh^2+ghb
\] 
serves as an entropy function, which contains the kinetic energy 
$\frac12h(u^2+v^2)$ and the gravitational potential energy 
$\frac12gh^2+ghb$. If the bottom topography $b$ and 
the solution of \eqref{swe} is smooth, 
a straightforward calculation \cite{FMT11} reveals that 
\begin{align}
  \label{entro}
 E(U)_t  
 + \left(
 \frac12(hu^3+huv^2)
 +ghu(h+b)
 \right)_x+\left(
 \frac12(hu^2v+hv^3)
 +ghv(h+b)
 \right)_y = 0,
\end{align}
which is obtained by taking the inner product of 
the SWEs \eqref{swe} with the 
{\it entropy variable} 
\begin{align}
  \label{entro-v}
  V:=\partial_U E = [g(h+b)-\frac12(u^2+v^2), u, v]^T
\end{align}
and applying the chain rule.
Since entropy should be dissipated across shock discontinuities, 
the entropy conservation equation \eqref{entro} needs to be replaced by the following 
entropy dissipation postulate
\begin{align}
  \label{entroD}
 E(U)_t  
 + \left(
 \frac12(hu^3+huv^2)
 +ghu(h+b)
 \right)_x+\left(
 \frac12(hu^2v+hv^3)
 +ghv(h+b)
 \right)_y \le 0,
\end{align}

\textbf{(ii) Steady states and well-balanced property}.
Another important issue which arises in SWEs \eqref{swe} is the simulation of
their steady states, which are solutions that are independent of the time
variables. The most import example of a steady state solution to \eqref{swe}
is the so-called {\it lake at rest}, given by 
\begin{align}
  \label{wb}
u=v\equiv 0, \quad h+b \equiv Const.
\end{align}
Many interesting applications, such as 
waves on a lake or tsunami waves in deep ocean, 
involve computing perturbations of the lake at rest.
A numerical scheme which preserves a discrete version of the steady state 
\eqref{wb}
is termed {\it well-balanced} with respect to the steady state.

\textbf{ (iii) Positivity of the water height}.
The water height $h$ in the SWEs needs to remain positive (non-negative)
for the system \eqref{swe} to remain well-posed.
If the water height  becomes negative, the system \eqref{swe} will be
non-hyperbolic and non-physical, and the problem will be ill-posed.

\textbf{(iv) The conservation property}.
For constant topography $b\equiv Const$, the balance law \eqref{bal}
reduces to a hyperbolic system of conservation laws 
\[
  U_t + F(U)_x + G(U)_y=0.
\]
Integrating the above equation over any control volume $D\subset \mathbb{R}^2$
and applying the Gauss law, there holds  the following conservation property:
\[
  \frac{d}{dt}\int_{D}U\,\mathrm{dx} = -\int_{\partial D}
  [F(U),  G(U)]^T\cdot \bld n\, \mathrm{ds},
\] 
where $\bld n$ is the outward unit normal direction on the boundary 
$\partial D$.

The main focus of this work is to construct high-order numerical schemes for the
SWEs \eqref{swe} on unstructured meshes that respect the above mentioned four
properties.
All these properties are important in practice:
\begin{itemize}
  \item  
The Lax-Wendroff theorem \cite{LW60} ensures that if 
a conservative scheme produces a discrete solution
$U_h(x,t)$ that uniformly converges to $\bar{U}(x,t)$, then ${\bar{U}}(x,t)$
is a weak solution to the continuous equation. Non-conservative schemes may converge to wrong solutions.
\item 
Many shallow water applications involve rapidly moving interfaces between wet
and dry areas, where no water is present. 
If no special attention is paid to maintain the positivity (non-negativity) 
of the water height,
standard numerical methods may produce unacceptable negative water height 
near the dry/wet front, which crashes of the numerical
simulation. 
\item Well-balanced schemes are essential for computing perturbations of steady
  states.
\item Entropy stability \eqref{entroD} provides additional stabilization
  mechanism to the
  scheme which further enhance its robustness. 
\end{itemize}
Various numerical schemes satisfying (part of) these properties 
for hyperbolic conservation laws or balance laws have been proposed in the
literature. We refer to the review articles \cite{XS14, Xing17} for a survey of
numerical schemes for the SWEs, in particular high-order well-balanced and
positivity-preserving schemes; to the review articles \cite{ZS11, XZ17}
for a survey of maximum-principle-satisfying and positivity-preserving
high-order schemes for conservation laws; and to the review article
\cite{Tadmor16} for entropy stable schemes.

Of particular relevance to the current work is the class of 
entropy stable schemes for the SWEs, which respect the entropy dissipation 
postulate
\eqref{entroD}. First-order entropy stable finite volume (FV) schemes for the SWEs were 
proposed in \cite{TadmorZhong08,FMT09, FMT11} where the key concepts of entropy
variable, entropy conservative/stable numerical fluxes were discussed.
Similar entropy conservative/stable numerical fluxes were adopted in the
high-order {\it nodal} DG literature, which,  in combination with 
discrete derivative operators
using Gauss-Lobatto quadrature points that
satisfy the summation-by-parts (SBP) property, 
 yield entropy conservative/stable DG discretizations 
\cite{GWK16, WWAGK17, WWAGW18}. 
High-order entropy stable DG schemes were more recently extended to {\it modal}
formulations \cite{WDGX20, WKC21}, following the work of Chan \cite{Chan18}.
All these works are also well-balanced and conservative.
These entropy stable spatial discretizations were then combined with 
{\it explicit} strong
stability preserving  Runge-Kutta (SSP-RK) time integrators \cite{GST01}
to yield fully discrete conservative and well-balanced schemes. 
We note that, in all these works,  the entropy stability was proven in the
semi-discrete level where only  spatial discretization was involved, which 
does not hold theoretically  for the explicit fully discretizations.
We further note that 
the schemes  \cite{WWAGW18, WDGX20} can preserve the positivity of the
water height with the aid of a positivity preserving scaling limiter \cite{ZS11}.
Moreover, most of the above cited works use structured/rectangular meshes, 
with the exceptions of \cite{WWAGW18} which uses unstructured quadrilateral
meshes and \cite{WDGX20} which works on unstructured triangular meshes.

In this work, we construct high-order locally conservative, positivity
preserving, well-balanced, and entropy stable DG schemes for the SWEs on general
triangular meshes. Our approach is very different from the above cited entropy
stable DG schemes which relies on the SBP property of the underling difference
operators. 
In our semi-discrete scheme,
instead of directly approximating the conservative variables, 
we use the water height and velocity as the solution unknowns.
As a result, entropy stability is achieved naturally within the
weak formulation without the need to convert to the 
strong form or work with difference
operators/matrices. 
Our approach is more closely related to the entropy variable approach
to entropy stable schemes proposed by Tadmor \cite{Tadmor86} back in 1986, as
the velocity is part of the entropy variable \eqref{entro-v}.
For the purpose of efficient explicit time integration, we still keep 
the water height as the solution unknown. As a result, we need to use the skew-symmetric formulation of the momentum equation 
\cite{GWK16} to guarantee entropy stability of the semi-discrete scheme.
Similar to the works \cite{WDGX20, WKC21}, our proposed scheme achieves 
entropy stability regardless of the choice of underlying numerical integration rules,
although the proofs are very different.
Actually, all integrals in our scheme involve polynomials only, which can be easily computed exactly if one wishes.

The proposed DG semi-discretization is then carefully combined with a classical SSP-RK
time integrator \cite{GST01}, in combination with a positivity-preserving
scaling limiter to ensure positivity of the water height. 
Here special attention is paid to the Runge-Kutta inner stage reconstructions to
maintain local conservation of the fully discrete scheme.
To do so, we re-introduce the discharge (momentum) 
as an auxiliary solution unknown and 
reconstruct inner stage values based on the 
the conservative variables, i.e.,
water height and discharge. We prove that this water height-velocity-discharge three-field formulation
is mathematically equivalent to the water height-velocity two-field formulation
in the semi-discrete level.
The advantage of this three-field formulation over the two-field formulation is that 
 standard slope limiting procedures can now be applied on 
the conservative variables (water height and discharge) 
to suppress numerical
 oscillations 
 near discontinuities
without violating the local conservation property. 
Here we apply the characteristic-wise TVB limiter \cite{CS98}  with the 
Fu-Shu troubled cell indicator \cite{FS17} using the total height as the indicating function.

The last ingredient of our fully discrete scheme is a proper wetting/drying treatment for problems with (moving) dry areas. The above mentioned positivity preserving limiter and TVB limiter do not directly work on the velocity approximation. As a result, the scheme may produce arbitrarily large velocity approximations near dry cells where the water height is very small.
Without any special treatment near these regions, the large velocity near dry areas will dictate the time step size, and may even crash the code due to too large velocity values.
We looked into a couple of wetting/drying treatments in the literature, but didn't find a good one yet that works for our scheme. Hence we introduce a new wetting/drying treatment that at least works for our numerical examples; see details in Remark \ref{rk:dry} below.

The rest of the paper is organized as follows. In Section 2, we introduce 
the reformation of SWEs \eqref{swe}, and  used it
to design a conservative, well-balanced and entropy stable DG spatial discretization.
In Section 3, 
we present the explicit temporal discretization, and then prove the 
posivitity preservation property. The implementation of a characteristic-wise 
TVD slope limiter with an efficient troubled cell indicator 
is then discussed. We further remark on the proper wetting/drying treatment in the velocity calculation.
Numerical results in one- and two-dimensions are then reported 
in Section 4. We draw concluding remarks in Section 5.

\section{Reformulation of SWEs and the DG semi-discretization}
In this section, we first reformulate the SWEs \eqref{swe} into an equivalent
skew-symmetric form, c.f. \cite{GWK16}, and then introduce the associated conservative, entropy stable and
well-balanced DG semi-discretization. Although using the same skew-symmetric
form, we emphasis that our entropy stable DG discretization is completely
different from the work \cite{GWK16}, where we use velocity as independent
solution unknowns.

\subsection{The skew-symmetric form of the SWEs}
Multiplying the mass conservation equation \eqref{swe1} by $\frac12\bld u$, and 
subtract it from the momentum balance equation \eqref{swe2}, we get
the following equivalent form of the SWEs:
\begin{subequations}
  \label{sweS}
  \begin{align}
  \label{sweS1}
  h_t +\nabla\cdot(h\bld u)=&\;0,\\
  \label{sweS2}
    (h\bld u)_t+ \nabla\cdot (h\bld u\otimes \bld u)
    +g h\nabla (h+b)
    -\frac12 h_t\bld u-\frac12\nabla\cdot(h\bld u)\bld u
    =&\;0.
\end{align}
\end{subequations}
Here \eqref{sweS2} is referred to as the skew-symmetric form of the momentum
balance equation \eqref{swe2}, c.f. \cite{GWK16}. 
Multiplying \eqref{sweS1} with $g(h+b)$ and \eqref{sweS2} with $\bld u$ and
adding, we immediate get the entropy conservation equality \eqref{entro}.
This suggests to use finite elements to directly approximate the quantities 
$g(h+b)$ and $\bld u$ in order to design a Galerkin method that respect the
entropy conservation property \eqref{entro}, which is the approach we take in this article.
In practice, we use a discontinuous finite element space to directly  approximate
the water height
$h(x,t)$ and use the same finite element space to approximate the 
bottom topography $b(x)$, so that $g(h+b)$ can be taken as a test function in the 
Galerkin formulation.

\subsection{The conservative, entropy-stable, and well-balanced DG spatial
discretization}
Without loss of generality, we formulate the DG spatial discretization for 
the SWEs \eqref{sweS} on a periodic domain $\Omega\subset \mathbb{R}^2$.
Other standard boundary conditions will be used in the numerical experiments.
Here we formulate the scheme on a general unstructured triangular mesh, while
noting that the proposed method works on any standard meshes.

To this end, let $\Omega_h:=\{K\}$ be a conforming triangular discretization of the
domain $\Omega$.  Denote $\partial \Omega_h:=\{\partial K\}$ as the
collection of element boundaries with $\bld n_K$ the associated outward unit
normal direction. Let $\mathcal{E}_h:=\{F\}$ be the collection
of edges of the triangulation $\Omega_h$. For any polynomial degree $k\ge 0$, let 
\begin{align}
  \label{space}
  V_h^k :=\{v\in L^2(\Omega):\quad v|_K\in\pol_k(K), \quad \forall
  K\in\Omega_h\},
\end{align}
where $\pol_k(K)$ is the space of polynomials of degree at most $k$ on the
element $K$.
Furthermore, let $\bld V_h^k$ be the  vectorial version of the space $V_h^k$.
Given an edge $F=K^+\cap K^-\in\mathcal{E}_h$ which is shared by two elements $K^+$
and $K^-$, we denote $\bld n$ as the unit normal direction on $F$ pointing
towards $K^-$, and denote 
$\jmp{\phi}|_F :=\phi^+-\phi^-$ and 
$\avg{\phi}|_F:=\frac12(\phi^++\phi^-)$ as the standard jump and average on $F$ 
for any function $\phi\in V_h^k$, where $\phi^\pm:=\phi|_{K^\pm}$. 

The proposed DG spatial discretization of \eqref{sweS} on the periodic domain
$\Omega$ reads as follows:
find $(h_h,\bld u_h)\in V_h^k\times \bld V_h^k$ such that
\begin{subequations}
  \label{dg}
  \begin{align}
  \label{dg1}
    M_h\left((h_h)_t, e_h\right) + A_h((h_h, \bld u_h), e_h) &\;= 0,\\
  \label{dg2}
    \bld M_h\left((h_h\bld u_h)_t, \bld v_h\right) 
    + B_h((h_h, \bld u_h), \bld v_h) 
    +C_h((h_h, \bld u_h), \bld v_h) \;\;&\;\nonumber\\
-
    M_h\left((h_h)_t, \frac12\bld u_h\cdot\bld v_h\right) - A_h\left((h_h, \bld u_h), \frac12\bld u_h\cdot\bld v_h\right) &\;= 0,
  \end{align}
for all $(e_h, \bld v_h)\in V_h^k\times \bld V_h^k$,
where 
$b_h\in V_h^k$ is a proper approximation of the bottom topography, and
the associated operators are given below:
\begin{align}
\label{op-1}
  M_h((h_h)_t, e_h):=   & \sum_{K\in\Omega_h}\int_K(h_h)_te_h\,\mathrm{dx},\\
  A_h((h_h, \bld u_h), e_h):=  & 
  \sum_{K\in\Omega_h}\left(-\int_Kh_h\bld u_h\cdot\nabla e_h\,\mathrm{dx}+
  \int_{\partial K}\widehat{h_h\bld u_h}\cdot \bld
n_Ke_h\,\mathrm{ds}\right),\nonumber\\
\label{op-2}
= &-\sum_{K\in\Omega_h}\int_Kh_h\bld u_h\cdot\nabla e_h\,\mathrm{dx}+
    \sum_{F\in\mathcal{E}_h}\int_F\widehat{h_h\bld u_h}\cdot \bld
    n\jmp{e_h}\,\mathrm{ds} ,\\
  \label{op-5}
    \bld M_h\Big((h_h\bld u_h)_t,\bld v_h\Big) 
:=  &\;
    \sum_{K\in\Omega_h}\int_K(h_h\bld u_h)_t
    \cdot\bld v_h\,\mathrm{dx},\\    
  B_h((h_h, \bld u_h), \bld v_h):=  & 
    -\sum_{K\in\Omega_h}\int_K
    h_h(\bld u_h\otimes\bld u_h):\nabla \bld v_h 
\,\mathrm{dx}+
\sum_{F\in\mathcal{E}_h}
\int_F
\widehat{(h_h\bld u_h\bld u_h)}\bld n
\cdot \jmp{\bld v_h}\,\mathrm{ds}, \label{op-3}\\
    C_h((h_h, \bld u_h), \bld v_h) := &\;
    \sum_{K\in\Omega_h}\int_Kgh_h\nabla(h_h+b_h)\cdot\bld v_h
\,\mathrm{dx}
-
\sum_{F\in\mathcal{E}_h}\int_Fg\jmp{h_h+b_h}\avg{h_h\bld v_h}\cdot\bld n
\,\mathrm{dx},
\label{op-4}
\end{align} 
where the numerical fluxes $\widehat{h_h\bld u_h}\cdot\bld n$
and $\widehat{(h_h\bld u_h\bld u_h)}\bld n$
in the operators \eqref{op-2} and
\eqref{op-3} are defined as follows:
\begin{align}
  \label{flux}
\widehat{h_h\bld u_h}\cdot\bld n:= &
\avg{h_h\bld u_h}\cdot\bld n + \frac12\alpha_h \jmp{h_h+b_h},\\
  \label{flux2}
  \widehat{(h_h\bld u_h\bld u_h)}\bld n:= &
\avg{h_h\bld u_h}\cdot\bld n\,\avg{\bld u_h} + 
\frac12\alpha_h \jmp{(h_h+b_h)\bld u_h},
\end{align}
with estimated maximum speed 
\begin{align}
  \label{speed}
  \alpha_h|_F :=\max\left\{\sqrt{gh_h^+}+|\bld u_h^+\cdot \bld n|,
  \sqrt{gh_h^-}+|\bld u_h^-\cdot \bld n|\right\}.
\end{align}
We will show below that these local Lax-Friedrichs type 
numerical fluxes are entropy-stable.
We note that the above operators are very natural  DG scretizations of  
the corresponding PDE operators in \eqref{sweS}, in particular,
\begin{itemize}
  \item The operators $M_h$ in \eqref{op-1} and $\bld M_h$ in \eqref{op-5} are the weak forms associated with  
    the time derivative term $h_t$ in \eqref{sweS1}, and 
    $
    (h_h\bld u_h)_t $ in \eqref{sweS2}, respectively;
  \item The operator $A_h$ in \eqref{op-2}
    is the DG discretization of the convection term 
    $\nabla\cdot(h\bld u)$    in 
    \eqref{sweS1}, with the numerical flux \eqref{flux}, and 
     The operator $B_h$ in \eqref{op-3} is the DG discretization of the convection term $\nabla\cdot(h\bld u\otimes\bld u)$
 in \eqref{sweS2}, with the numerical flux \eqref{flux2}. 
 Here the particular choice of the numerical fluxes \eqref{flux} and \eqref{flux2}  is crucial for the {\it entropy stability} of the semi-discrete scheme \eqref{dg}. Similar numerical fluxes have been used in the literature, c.f. \cite{FMT11,GWK16};
\item The operators $M_h$ and $A_h$ in \eqref{dg2} are the operators associated with the  
skew-symmetric terms $\frac12 h_t\bld u$ and $\frac12\nabla\cdot(h\bld u)\bld u$ in \eqref{sweS2}, respectively;
\item The operator $C_h$ in \eqref{op-4} is a DG discretization of 
  the 
gravitational
  term $gh\nabla(h+b)$ in \eqref{sweS2} using a central numerical flux.
  To see this, we note that the DG discretization with central numerical flux
  for this operator reads as follows:
  \begin{align*}
    - \sum_{K\in\Omega_h}\int_Kg(h_h+b_h)\nabla\cdot(h_h\bld v_h)
\,\mathrm{dx}
+
\sum_{F\in\mathcal{E}_h}\int_Fg\avg{h_h+b_h}\jmp{h_h\bld v_h}\cdot\bld n
\,\mathrm{dx},
  \end{align*}
which is equivalent to $C_h$ by integration by parts.
We mention that the two gravitational terms in momentum balance \eqref{swe2} 
are combined into a single non-conservative product, which is another key to 
the entropy stability and well-balanced property of our scheme.
We note that such  non-conservative  product has been explored in the literature, e.g., 
\cite{RBV08}, to design
well-balanced DG schemes.
\end{itemize}
\end{subequations}

We conclude this section with the main properties of our proposed DG
discretization \eqref{dg}, namely, local conservation, entropy stability and
the well-balanced property.
\begin{theorem}
\label{thm1}
  The semi-discrete scheme \eqref{dg} is 
\begin{subequations}
  \label{semi-p} 
  \begin{itemize}
    \item locally conservative in the sense of the following equalities:
    \begin{align}
      \label{e1}
\frac{d}{dt} \int_K h_h\,\mathrm{dx} = &\; -\int_{\partial K}\widehat{h_h\bld u_h}\cdot\bld n_K\,\mathrm{ds}, \\
      \label{e2}
\frac{d}{dt} \int_K h_h\bld u_h\,\mathrm{dx} = &\; 
-\int_{\partial K}  \widehat{(h_h\bld u_h\bld u_h)}\bld n_K\,\mathrm{ds}\\
&\;-
\int_{K}  gh_h\nabla(b_h)\,\mathrm{dx}
+\int_{\partial K}  \frac12gh_hb_h\bld n_K\,\mathrm{ds}\nonumber
    \end{align}
  \item entropy stable in the sense of the following equality:
\begin{align}
  \label{e3}
  \frac{d}{dt}E_h = &\; 
  -\sum_{F\in\mathcal{E}_h}\int_{F}\frac12\alpha_h
  \left(g\jmp{h_h+b_h}^2+\avg{h_h+b_h}\jmp{\bld u_h}\cdot\jmp{\bld u_h}\right)\,\mathrm{ds}
  \le 0 ,
\end{align}  
where the discrete entropy (total energy) is  
\[
  E_h:= \sum_{K\in\Omega_h} \int_K(\frac12\bld h_h|\bld
  u_h|^2+\frac12gh^2+ghb)\mathrm{dx},
\] 
\item well-balanced in the sense that it preserve the lake-at-rest steady state:
if the initial condition satisfies
\begin{align}
\label{eX}
  \bld u_h(0) = 0, \text{and}\quad  h_h(0)+b_h = C,
\end{align}
the solution to the semi-discrete scheme \eqref{dg} satisfies
\begin{align}
  \label{e4}
  \bld u_h(t) = 0, \text{and}\quad h_h(t)+b_h = C.
\end{align}
\end{itemize}
\end{subequations}
\end{theorem}

\begin{proof}
Taking test function $e_h:=1$ on element $K$ and {\it zero} elsewhere in \eqref{dg1}, we get mass conservation in \eqref{e1}.
Denote $u_h$ and $v_h$ as the two components of the velocity approximation $\bld u_h$.
Taking test function $e_h:=\frac12u_h$ on element $K$ and {\it zero} elsewhere in \eqref{dg1}, and $\bld v_h:=(1,0)$ on element $K$ and {\it zero} elsewhere in \eqref{dg2}
and adding, we get the following:
\begin{align*}
\frac{d}{dt} \int_K h_hu_h\,\mathrm{dx}   =
-\int_{\partial K}  \widehat{(h_h\bld u_h\bld u_h)}n_x\,\mathrm{ds}-
\int_{K}  gh_h\partial_x(h_h+b_h)\,\mathrm{dx}
+\int_{\partial K}  \frac12gh_h(h_h+b_h) n_x\,\mathrm{ds},
\end{align*}
where $n_x$ is the first component of the normal direction $\bld n_K$.
Combining the above identity with the fact that 
\begin{align*}
\int_{K}  gh_h\partial_x(h_h)\,\mathrm{dx}
-\int_{\partial K} \frac12 gh_h^2 n_x\,\mathrm{ds}
=0,
\end{align*}
we get the first component of the momentum balance identity \eqref{e2}.
We can apply the same argument to obtain the second component of the momentum balance identity \eqref{e2}.

Taking test function $e_h=g(h_h+b_h)$ in \eqref{dg1} and 
$\bld v_h=\bld u_h$ in \eqref{dg2} and adding, we get
\begin{align*}
&\;    M_h((h_h)_t, g(h_h+b_h))
    +\bld M_h((h_h\bld u_h)_t, \bld u_h)
    -\frac12M_h((h_h)_t, \bld u_h\cdot\bld u_h)\\
=&\; \underbrace{-B_h((h_h, \bld u_h), \bld u_h)
    + A_h\Big((h_h, \bld u_h), \frac12\bld u_h\cdot\bld u_h\Big)}_{:=I_1} \underbrace{-C_h((h_h, \bld u_h), \bld u_h) -
   A_h\Big((h_h, \bld u_h), g(h_h+b_h)\Big)}_{:=I_2}.
\end{align*}
Simplifying the above equality, we  yield the entropy dissipation equality \eqref{e3}.
More specifically, it is easy to show that the left hand side of the above equality  is 
the entropy dissipation rate
$\frac{d}{dt}E_h$, and the first term in the above right hand side
\begin{align}
\label{i1}
    I_1 = -\sum_{F\in\mathcal{E}_h}\int_{F}\frac12\alpha_h\avg{h_h+b_h}\jmp{\bld u_h}\cdot\jmp{\bld u_h}\,\mathrm{ds},
\end{align}
and the second term
\begin{align*}
    I_2 = -\sum_{F\in\mathcal{E}_h}\int_{F}\frac12\alpha_hg\jmp{h_h+b_h}^2\,\mathrm{ds},
\end{align*}
Below we only give detailed proof of the identity for $I_1$.
We have 
\begin{align*}
I_1 = &\;    -B_h((h_h, \bld u_h), \bld u_h)
    + A_h\Big((h_h, \bld u_h), \frac12\bld u_h\cdot\bld u_h\Big)\\
    = &\;   
    \underbrace{ \sum_{K\in\Omega_h}\int_K
  \left(  h_h(\bld u_h\otimes\bld u_h):\nabla \bld u_h 
  -\frac12 h_h\bld u_h\cdot \nabla (\bld u_h\cdot\bld u_h)
\right)\,\mathrm{dx}}_{\equiv 0} \\
&\;-\sum_{F\in\mathcal{E}_h}
\int_F
\left(\widehat{(h_h\bld u_h\bld u_h)}\bld n
\cdot \jmp{\bld u_h}
-\widehat{h_h\bld u_h}\cdot\bld n
\jmp{\frac12\bld u_h\cdot\bld u_h}
\right)
\,\mathrm{ds}.
\end{align*}
Now by the definition of the numerical fluxes in \eqref{flux}--\eqref{flux2} and the 
simple fact that $\jmp{ab} = \avg{a}\jmp{b}+\jmp{a}\avg{b}$, we have 
\[
\widehat{(h_h\bld u_h\bld u_h)}\bld n
\cdot \jmp{\bld u_h}
-\widehat{h_h\bld u_h}\cdot\bld n
\jmp{\frac12\bld u_h\cdot\bld u_h} = \frac12\alpha_h\avg{h_h+b_h}\jmp{\bld u_h}\cdot\jmp{\bld u_h},
\]
which proves the identity for $I_1$ in \eqref{i1}.

Finally, under the assumption \eqref{eX}, it is trivial to show that the spatial operators $A_h, B_h, C_h$ in \eqref{dg} all stays {\it zero}.
Hence, $(h_h)_t \equiv 0$ from equation \eqref{dg1}, and 
$(h_h\bld u_h)_t = \frac12(h_h)_t\bld u_h \equiv 0$ from \eqref{dg2}.
This implies the well-balanced property \eqref{e4}.
\end{proof}

\begin{remark}[Comparison with other entropy stable DG schemes]
\label{rk:es}
Our first order scheme with polynomial degree $k=0$ is closely related to the first order finite volume entropy stable schemes \cite{FMT09, FMT11}, as both approaches use the concept of entropy conservative/stable fluxes, cf. \eqref{flux}--\eqref{flux2}. The work \cite{FMT09,FMT11} promote to use a Roe-type dissipation operator in the numerical flux, while our numerical dissipation is of the more dissipative Lax-Friedrichs type. 

For our high-order DG scheme with $k\ge 1$, we were not able to find similar work in the literature. There are two main approaches to construct high-order entropy stable schemes for nonlinear conservation laws, both stem from Tadmor's pioneer work on entropy variables and entropy conservative/stable fluxes \cite{Tadmor87, Tadmor86}. 
The first approach directly discretizes the conservation equations using the entropy variables \cite{Tadmor86}, see also \cite{Hughes86}. The major drawback of this approach is that explicit time stepping is usually not applicable to these schemes due to the highly nonlinear mapping between the entropy variables and the conservative variables.
Hence they are generally more expensive than other explicit schemes.
The second approach is based on the (quadrature-based) SBP operator concept, and has undergo a major development in the past few years, see, e.g., the entropy stable DG schemes \cite{GWK16, WWAGK17, WWAGW18,WDGX20, WKC21} for SWEs. These quadrature-based approaches may lead to accuracy loss, and they may be more cumbersome to implement on unstructured triangular meshes than classical DG schemes; see the more discussion in the recent review work \cite{CS20}.

Our scheme \eqref{dg} combines both advantages of the above mentioned approaches: 
\begin{itemize}
    \item the proof of entropy stability can be performed directly on the variational formation \eqref{dg} without converting to any matrix-vector form. 
This is made possible due to the use of velocity approximation and
the skew-symmetrization of the momentum balance equation. 
In particular, we have the entropy stability result \eqref{e3} for the DG scheme \eqref{dg} with any choice of numerical quadrature rule. 
This is possible because the proof of \eqref{e3} does not rely on integration by parts as the operators have already been properly skew-symmetrized.
Hence, the quadrature rule can be chosen only for accuracy considerations. In our numerical experiments, we simply use Gauss quadrature rules that are exact for integrating polynomials of degree $2k$. 
\item 
the semi-discrete scheme \eqref{dg} can be discretized in time using classical explicit time stepping schemes. This is possible because we discretize the water height $h_h$ as the solution unknown.
\end{itemize}
\end{remark}

\section{Fully discrete scheme: local conservation, well-balanced property,
positivity preservation, and slope limiting}
In this section, we discrete the semi-discrete scheme \eqref{dg} in time using  
explicit SSP-RK time integrators.
Special attention is paid to maintain the local conservation, well-balanceness,
and positivity preservation properties. 
We also discuss the use of a characteristic-wise TVB slope limiter
\cite{CS98} in combination with the recent troubled-cell indicator proposed in
\cite{FS17} to improve its efficiency. The slope limiter, which suppress
numerical oscillations near shock discontinuities, is a crucial component
for the accuracy and robustness of the overall scheme for polynomial degree $k\ge 1$, 
c.f. \cite{CS01}.
Moreover, we propose a simple wetting/drying treatment for the velocity calculation near dry cells where water height is small.

\subsection{A three-field reformulation of the semi-discrete scheme \eqref{dg}}
Here we introduce a three-field reformation of the semi-discrete DG scheme \eqref{dg} by using the discharge $\bld m:= h\bld u$ as an additional independent unknown, which is then discretized in time using the explicit SSP-RK method.
The three-field DG scheme reads as follows:
find $(h_h, \bld u_h, \bld m_h)\in V_h^k\times\bld V_h^k\times\bld V_h^k$ such that
\begin{subequations}
\label{dgr}  
  \begin{align}
      \label{dgr1}
    M_h\left((h_h)_t, e_h\right) + A_h((h_h, \bld u_h), e_h) = &\;0,\\
    \label{dgr2}
    \bld M_h\left((\bld m_h)_t, \bld v_h\right) 
    + B_h((h_h, \bld u_h), \bld v_h) 
    +C_h((h_h, \bld u_h), \bld v_h) \;\;\quad&\\
        -M_h\left((h_h)_t, \frac12\bld u_h\cdot\bld v_h\right) - A_h\left((h_h, \bld u_h), \frac12\bld u_h\cdot\bld v_h\right) = &\;0,\\
    \label{dgr3}
    \bld M_h\left(h_h\bld u_h-\bld m_h, \bld w_h\right) = &\;0,
  \end{align}
\end{subequations}
for all $(e_h, \bld v_h,\bld w_h)\in V_h^k\times\bld V_h^k\times \bld V_h^k$.
We have the following equivalence of the two formulations \eqref{dg} and \eqref{dgr}.
\begin{theorem}
Let $(h_h, \bld u_h, \bld m_h)$ be the solution to the three-field DG formulation \eqref{dgr}.
Then $(h_h, \bld u_h)$ is the solution to the two-field DG formulation \eqref{dg}.
\end{theorem}
\begin{proof}
Taking the time derivative of  equation \eqref{dgr3}, we can replace the auxiliary variable $\bld m_h$ in \eqref{dgr2} by $h_h\bld u_h$. This means the solution $(h_h,\bld u_h)$ solves the system \eqref{dg}.
\end{proof}
The advantage of this reformulation will be clear next when we discuss SSP-RK time discretizations and slope limiting.

\subsection{High order SSP-RK discretization and 
inner stage reconstruction}
The semi-discrete scheme \eqref{dgr} is not a standard ODE system  
$U_t + \mathcal{F}(U) = 0$, with $U$ being the solution vector and $\mathcal{F}(U)$ the spatial operator, as the time derivative terms in \eqref{dgr2}  involve the nonlinear product $h_h\bld u_h$, and \eqref{dgr3} is an algebraic equation. As a result, special care is need in design locally conservative high-order time discretizations. Here we apply the third order SSP-RK3 scheme to \eqref{dgr} which preserves the local conservation property. It is built on top of a plain forward Euler discretization in Algorithm \ref{alg1}, 
 a velocity update in Algorithm \ref{algV}, and
a convex combination step in Algorithm \ref{alg2}. The full plain SSP-RK3 algorithm without slope limiting is given in Algorithm \ref{alg3}.

\newcommand{\ALOOP}[1]{\ALC@it\algorithmicloop\ #1%
  \begin{ALC@loop}}
\newcommand{\ENDALOOP}{\end{ALC@loop}\ALC@it\algorithmicendloop}
\renewcommand{\algorithmicrequire}{\textbf{Input:}}
\renewcommand{\algorithmicensure}{\textbf{Output:}}
\newcommand{\algorithmicbreak}{\textbf{break}}
\newcommand{\BREAK}{\STATE \algorithmicbreak}
\begin{algorithm}
\caption{Plain Forward Euler + DG}
\begin{algorithmic}[1]
\REQUIRE $h_h^{old}\in V_h^k$, $\bld u_h^{old}, \bld m_h^{old}\in \bld V_h^k$, and time step size $\Delta t>0$.
\ENSURE $h_h^{new}\in V_h^k$, and $\bld m_h^{new}\in \bld V_h^k$.
  \STATE  Compute $h_h^{new}$ by the following equation:
  \[
  M_h(h_h^{new}, e_h) = M_h(h_h^{old}, e_h) - \Delta tA_h((h_h^{old}, \bld u_h^{old}), e_h),\quad \forall 
  e_h\in V_h^k,
  \]
  \STATE Compute $\bld m_h^{new}$ by the following equation:
  \begin{align*}
  \bld M_h(\bld m_h^{new}, \bld v_h) =&\; \bld M_h(\bld m_h^{old}, \bld v_h) - \Delta tB_h((h_h^{old}, \bld u_h^{old}), \bld v_h)
  - \Delta tC_h((h_h^{old}, \bld u_h^{old}), \bld v_h)\\
&\hspace{-1cm}  +\frac12M_h\left(h_h^{new}-h_h^{old}, \bld u_h^{old}\cdot\bld v_h\right) +\frac12\Delta tA_h\left((h_h^{old}, \bld u_h^{old}), \bld u_h^{old}\cdot\bld v_h\right),\quad \forall 
  \bld v_h\in \bld V_h^k,
  \end{align*}
\end{algorithmic}
\label{alg1}
\end{algorithm}

\begin{algorithm}
\caption{Velocity update}
\begin{algorithmic}[1]
\REQUIRE $h_h\in V_h^k$, $\bld m_h\in \bld V_h^k$.
\ENSURE $\bld u_h\in \bld V_h^k$.
\STATE Compute $\bld u_h$ by the following equation:
\[
 \bld M_h(h_h\bld u_h, \bld w_h) =\bld M_h(\bld m_h, \bld w_h),\quad \forall 
  \bld w_h\in \bld V_h^k.
  \]
\end{algorithmic}
\label{algV}
\end{algorithm}

\begin{algorithm}
\caption{Convex combination}
\begin{algorithmic}[1]
\REQUIRE Data $h_h^{1}, h_h^{2}\in V_h^k$, $\bld m_h^{1}, \bld m_h^{2}\in \bld V_h^k$. Positive weights $w_1, w_2$ with $w_1+w_2=1$.
\ENSURE $h_h\in V_h^k$, and $\bld m_h\in \bld V_h^k$.
  \STATE  Compute $h_h$ and $\bld m_h$ using convex combination:
  \[
h_h\leftarrow w_1h_h^1+w_2h_h^2, \quad
\bld m_h\leftarrow w_1\bld m_h^1+w_2\bld m_h^2.
  \]
\end{algorithmic}
\label{alg2}
\end{algorithm}

\begin{algorithm}
\caption{Plain SSP-RK3 + DG}
\begin{algorithmic}[1]
\REQUIRE $h_h^{n}\in V_h^k$, $\bld u_h^{n}, \bld m_h^{n}\in \bld V_h^k$ at time level $t^n$, and time step size $\Delta t>0$.
\ENSURE $h_h^{n+1}\in V_h^k$, and $\bld u_h^{n+1}, \bld m_h^{n+1}\in \bld V_h^k$ at next time level $t^{n+1}:=t^n+\Delta t$.
  \STATE  Apply Algorithm \ref{alg1} with inputs $h_h^n, \bld u_h^n, \bld m_h^n$ and $\Delta t$. 
  Denote the outputs as $h_h^{(1)}, \bld m_h^{(1)}$.
\STATE  Apply Algorithm \ref{algV} with inputs $h_h^{(1)}, \bld m_h^{(1)}$.
Denote the velocity output as $\bld u_h^{(1)}$. 
\STATE Apply Algorithm \ref{alg1} with inputs $h_h^{(1)}, \bld u_h^{(1)}, \bld m_h^{(1)}$ and $\Delta t$. Denote outputs as $h_h^{(2*)}, \bld m_h^{(2*)}$.
  \STATE 
  Apply Algorithm \ref{alg2} with inputs $h_h^{n}, h_h^{(2*)}, \bld m_h^{n}, \bld m_h^{(2*)}$ and weights $w_1=0.75, w_2 = 0.25$.
  Denote outputs as  $h_h^{(2)}, \bld m_h^{(2)}$.
 \STATE  Apply Algorithm \ref{algV} with inputs $h_h^{(2)}, \bld m_h^{(2)}$.
Denote the velocity output as $\bld u_h^{(2)}$. 
\STATE Apply Algorithm \ref{alg1} with inputs $h_h^{(2)}, \bld u_h^{(2)}, \bld m_h^{(2)}$ and $\Delta t$. Denote outputs as $h_h^{(3*)}, \bld m_h^{(3*)}$.  
 \STATE 
  Apply Algorithm \ref{alg2} with inputs $h_h^{n}, h_h^{(3*)}, \bld m_h^{n}, \bld m_h^{(3*)}$ and weights $w_1=1/3, w_2 = 2/3$.
  Denote outputs as  $h_h^{n+1}, \bld m_h^{n+1}$.
   \STATE  Apply Algorithm \ref{algV} with inputs $h_h^{n+1}, \bld m_h^{n+1}$.
Denote the velocity output as $\bld u_h^{n+1}$. 
\end{algorithmic}
\label{alg3}
\end{algorithm}

\begin{remark}[Local conservation and well-balanced property]
Similar to the proof of local conservation in Theorem \ref{thm1}, the 
forward Euler algorithm Algorithm \ref{alg1} is also locally conservative.
Meanwhile, the local conservation property is not polluted by the convex combination step in Algorithm \ref{alg2} due to the use of conservative variables in convex combination. 
If the velocity $\bld u_h$ were to be used in the convex combination step, local conservation for the discharge would be lost.
This is the major reason that the discharge $\bld m_h$ is re-introduced as an independent unknown in the DG formulation. Hence the overall algorithm 
Algorithm \ref{alg3} is locally conservative.
Moreover, it is easy to see that Algorithm \ref{alg3}
preserve the steady state solution \eqref{wb}, hence it is also well-balanced.
\end{remark}

\begin{remark}[Computational cost and positivity of water height]
The computational cost of Algorithm \ref{alg1} involves the evaluation of the right hand side operators, and solving the associated linear system for the mass matrix for $V_h^k$ in Step 1, 
and for $\bld V_h^k$ in Step 2.
The mass matrices can be made diagonal if orthogonal $L^2$-basis is used in the computation.
The computational cost of Algorithm \ref{algV} involves the linear system solve of a water height-weighted mass matrix, which is block diagonal and can be computed very efficiently.
The Algorithm \ref{alg2} is simply a vector update.
Hence, the computational cost in Algorithm \ref{alg3} is of linear complexity with respect to the total number of elements, which is similar to,  but slightly more expensive than (due to the velocity computation in Algorithm \ref{algV}), the cost of a classical DG scheme with SSP-RK3 time stepping.

Note that in Algorithm \ref{algV}, we need to invert the water height-weighted mass matrix to compute the velocity approximation $\bld u_h$.
This weighted mass matrix may fail to be invertible if the water height approximation $h_h$ become negative in parts of the domain.
In practice, requiring positivity of water height on the volume integration points used to compute these mass matrices suffice to ensure its invertibility, which, however, is not guaranteed in the plain Algorithm \ref{alg3}.
In the next subsection, we apply the positivity-preserving limiting approach used in \cite{XZ13} to guarantee such positivity requirement.
\end{remark}

\subsection{Hydrostatic reconstruction and posivitivity-preserving limiter}
The key idea of the posivity-preserving limiter in \cite{XZ13} is to ensure the cell average of water height is positive after one step of forward Euler time stepping under a usual CFL condition.
Given solution $h_h^n, \bld u_h^n$ at time $t^n$, and time step size $\Delta t$, the water height $h_h^{n+1}$ at next time level for the forward Euler Algorithm \ref{alg1} reads as follows:
\begin{align*}
  M_h(h_h^{n+1}, e_h) = M_h(h_h^{n}, e_h) - \Delta tA_h((h_h^{n}, \bld u_h^{n}), e_h),\quad \forall 
  e_h\in V_h^k.
\end{align*}
Taking test function $e_h=1$ on a single element $K\in\Oh$, we get the evolution equation for the water height cell average:
\begin{align}
\label{cell}
\bar h_K^{n+1} = \bar h_K^{n} - \frac{\Delta t}{|K|}\int_{\partial K}\widehat{h_h^n\bld u_h^n}\cdot\bld n_K\,\mathrm{ds}
\end{align}
where $\bar{h}_K^n$ stands for the cell average of $h_h$ on the triangle $K$ at time level $t^n$, and $|K|$ is the area of the element $K$.
Due to  the fact that the numerical flux \eqref{flux} contains the jump of bottom topography $b_h$, which can arbitrarily large, we can not prove positivity of $\bar h_K^{n+1}$ in the above equation \eqref{cell} under the condition of positivity of $\bar h_K^n$ and a reasonable time step size restriction. To fix this, we slightly modify the numerical fluxes using the idea of hydrostatic reconstruction \cite{ABF04,XZ13}. 
In particular,  introducing the following hydrostatic reconstructed version of the water height:
\begin{subequations}
\label{hydro}
\begin{align}
    h_h^{*,+} := \max\left\{0, h^+ +\min\{0, \jmp{b_h}\} \right\},\\
    h_h^{*,-} := \max\left\{0, h^- -\max\{0, \jmp{b_h}\} \right\},
\end{align}
\end{subequations}
we replace $h_h^{\pm}$ in the flux terms in the scheme \eqref{dg} 
by $h_h^{*,\pm}$, and replace the associated the jump term $\jmp{h_h+b_h}$ by $\jmp{h_h^{*}} = 
h_h^{*,+}-h_h^{*,-}$.
For example, the flux \eqref{flux} is now replaced by the following one:
\begin{align}
    \label{fluxH}
    \widehat{h_h^*\bld u_h}\cdot\bld n:=\avg{h_h^*\bld u_h}\cdot\bld n
    +\frac12\alpha_h^*\jmp{h_h^*},
\end{align}
with
\begin{align}
\label{speed2}
  \alpha_h^*|_F :=\max\left\{\sqrt{gh_h^{*,+}}+|\bld u_h^+\cdot n|,
  \sqrt{gh_h^{*,-}}+|\bld u_h^-\cdot n|\right\}.
\end{align}
It is clear that if $h_h$ satisfies $h_h+b_h=Const$ with $h_h>0$, then 
\begin{align*}
    h_h^{*,+} = h^+ +\min\{0, \jmp{b_h}\},\\
    h_h^{*,-} =  h^- -\max\{0, \jmp{b_h}\},
\end{align*}
and $\jmp{h_h^*} \equiv \jmp{h_h+b_h}$, which implies that the modified fluxes will not pollute the well-balanced property of the original fluxes.

With this modification, the forward Euler discretization lead to the following cell average evolution for water height:
\begin{align}
\label{cell2}
\bar h_K^{n+1} =&\; \bar h_K^{n} - \frac{\Delta t}{|K|}\int_{\partial K}\widehat{h_h^{*,n}\bld u_h^n}\cdot\bld n_K\,\mathrm{ds}\nonumber\\
=&\; \bar h_K^{n} - \frac{\Delta t}{|K|}\int_{\partial K}
\left(\avg{h_h^{*,n}\bld u_h^{n}}\cdot\bld n_K + \frac12\alpha_h^{*,n}\jmp{h_h^{*,n}}_K\right)
\,\mathrm{ds},
\end{align}
where $\jmp{h_h}_K:= h_h^{int(K)} - h_h^{ext(K)}$ is the jump, with 
$h_h^{int(K)}$ and $h_h^{ext(K)}$ being the approximations 
obtained from the interior and the exterior of $K$.
Note that by definition, on any edge $F=K^+\cap K^-$ shared by two elements, there holds 
\begin{align*}
    \jmp{h_h} = h_h|_{K+}-h_h|_{K^-} = \jmp{h_h}_{K^+}  = -\jmp{h_h}_{K^-}.
\end{align*}
The cell average evolution equation \eqref{cell} now has a similar form as 
\cite[Equation 3.1]{XZ13}. Hence, we can follow the same analysis in \cite[Section 3]{XZ13} to ensure positivity of the water height cell average at next time level. 
The following result is Theorem 3.2 in \cite{XZ13}. The proof is almost identical, hence we omit it for simplicity.
\begin{theorem}[Theorem 3.2 in \cite{XZ13}]
\label{thm:pp}
For the scheme \eqref{cell2} to be positivity preserving, i.e., $\bar h_K^{n+1}\ge 0$, a sufficient condition is that $h_K(\bld x)\ge 0$, 
$\forall \bld x\in S_K$ for all $K$, under the CFL condition
\begin{align}
\label{cfl}
    \alpha \frac{\Delta t}{|K|}|\partial K| \le \frac23\widehat w_1.
\end{align}
Here 
$h_K(\bld x)$ denotes the polynomial for water height at time level $n$,
$S_K$ is a set of (symmetric) quadrature points on $K$ that includes 
$k+1$ Gauss quadrature points on each boundary edge,
$\alpha$ is the maximum estimated speed \eqref{speed2}, 
$|\partial K|$ is the perimeter of element $K$, and 
$\widehat w_1$ is the quadrature weight of the $\lceil \frac{k+3}{2}\rceil$-point Gauss-Lobatto rule on
$[-1/2, 1/2]$ for the first quadrature point.
\end{theorem}

At time level $n$, given the water height DG polynomial $h_K(\bld x)$ with its cell average $\bar h_K^n\ge0$, we use the simple scaling limiter \cite[Section 3.4]{XZ13} to ensure the above sufficient condition $h_K(\bld x)\ge0$ for all 
$\bld x\in S_K$, i.e., replacing 
$h_K(\bld x)$ by a linear scaling around the cell average:
\begin{align}
\label{pplimit}
    \widetilde{h}_K(\bld x) = \theta_K(h_K(\bld x)-\bar h_K^n)+\bar h_K^n,
\end{align}
where $\theta_K\in[0,1]$ is determined by 
\begin{align}
\label{scale}
    \theta_K := \min_{\bld x\in S_K}\theta_{\bld x},
    \quad\theta_{\bld x} =\min\left\{1, 
    \frac{\bar h_K^n }{\bar h_K^n-h_K(\bld x)}
    \right\}.
\end{align}
A slightly more efficient and less restrictive scaling parameter $\theta_K$ can be obtained using a reduced set of quadrature points, see \cite[Section 3.4]{XZ13} for more details.
In practice, the positivity preserving limiter \eqref{pplimit} is applied in each inner stage of the RKDG algorithm \ref{alg3}.
We notice that the well-balanced property is also not affected by this positivity preserving limiter.

\begin{remark}[On hydrostatic reconstruction and bottom topography approximation]
The proof of Theorem \ref{thm:pp} requires the use of hydrostatic reconstruction \eqref{hydro}, which is needs due to the lack of control of the bottom topography jump $\jmp{b_h}$ across edges. 
When the polynomial degree $k\ge 1$ in the DG scheme \eqref{dg}, 
one can approximate the bottom topography using a continuous approximation 
$b_h\in V_h^k\cap H^1(\Omega)$, which implies $\jmp{b_h}=0$. 
In this case, under the positivity assumption of Theorem \ref{thm:pp}, we have 
$h^{*,\pm}_h = h^{\pm}_h$, hence equivalence of the original scheme \eqref{cell} and the reconstructed version \eqref{cell2}.
For this reason, we prefer to use a continuous bottom topography approximation for $k\ge 1$, where the hydrostatic reconstruction \eqref{hydro} is not necessary anymore.
\end{remark}

\subsection{The troubled-cell indicator and slope limiter}
Another important ingredient of the DG methods is the slope
limiter procedure which is needed to suppress spurious oscillations near solution discontinuities.
We follow the standard slope limiting procedure in 
RKDG methods \cite{CS01, QS05}: 
\begin{itemize}
    \item [(1)] First we identify the {\it troubled cells}, namely, those cells which might need the limiting procedure.
    \item [(2)] Second we replace the solution polynomials in those troubled cells by reconstructed polynomials
    with limited slopes that  maintain the original
cell averages (conservation).
\end{itemize}
We use the Fu-Shu troubled-cell indicator proposed in \cite{FS17} to identify the troubled cells, with a scaling modification to improve its performance and computational efficiency. 
Given a discontinuous function $p\in V_h^k$, the 
troubled-cell indicator \cite{FS17} is given as follows:
\begin{align}
    \label{indicator}
    I_{K}(p) = \frac{\sum_{T\in \omega(K)}|\bar{\bar p}_T-\bar p_K|}{\bar p_{\max}-\bar p_{\min}},
\end{align}
where $\omega(K)$ is the union of cells that share a common edge with 
$K$, including $K$ itself, and $\bar{\bar p}_T$ is the cell average of the polynomial $p|_T$ extended to the target cell $K$, and 
$\bar p_{\max}$ and $\bar p_{\min}$ are the global maximal and minimal cell average on the domain.
Relying on the {\it assumption} \cite{CS98, CS01} that spurious oscillations are present in the solution $p_h$ only if they are present in its linear part $p_h^1$, which is its $L^2$-projection into the space of piecewise linear functions 
$V_h^1$, we use use the linear $L^2$-projection of the total height $h_h+b_h$ as the indicating function in \eqref{indicator}, which simplifies the implementation of the extended cell average $\bar{\bar h}_h$  for high-order case where the polynomial degree $k>1$. The cell $K$ is marked as a troubled cell if 
\begin{align}
    \label{tol}
    I_K(h_h^1+b_h^1) > tol,
\end{align}
where $tol$ is a user defined parameter. 
Note that this indicator is of $\mathcal{O}(h^2)$ in smooth regions, and of $\mathcal{O}(1)$ near discontinuities, hence is expected to be effective to detect troubled cells near discontinuities with a proper choice of $tol$.
Our numerical experiments suggest that the indicator is not too sensitive to the tolerance $tol$. The indicator with  $tol\in (0.01, 0.1)$ performs similarly for most of the examples, where $tol=0.01$ leads to a slightly larger number of detected troubled cells than $tol=0.1$ as expected. In our implementation, we take $tol=0.02$ for all the reported results.

\begin{remark}[On scaling of the indicator \eqref{indicator}]
The original indicator proposed in \cite{FS17} use the local maximal cell average $\max_{T\in \omega(K)}\{\bar p_T\}$ as the scaling denominator. This scaling has the drawback of not able to detect any troubled cells for small perturbation tests where the total height is a very small perturbation of a constant state. In particular, the original indicator with a local maximum scaling will produce a completely different result when the indicating function is perturbed by a global constant $p(x)\leftarrow p(x)+Const$. 
The new global difference  scaling denominator $\bar p_{\max}-\bar p_{\min}$ now produce the same indicating value when the indicating function is perturbed by a global constant. It performs quite well for all the numerical examples reported here.
We further mention that this global scaling is suggested to us by Prof. Chi-Wang Shu from Brown University in a private communication.
\end{remark}


After the troubled cells have been detected, we apply the characteristic-wise TVB limiter \cite{CS98, CS01} on the 
conservative variables $(h_h+b_h, \bld m_h)$
with TVB parameter $M=0$. 
To save space, 
we leave out the derivation of this limiter and refer to \cite{CS98, CS01} for details.
We mention that while this TVB limiter is compatible with the well-balanced property of the DG scheme as $h_h+b_h$ is used in the limiting process.
In practice, we first apply this TVB limiter then apply the positivity preserving limiter \eqref{pplimit} for each inner Runge-Kutta stage values.

\subsection{Velocity computation and dry cell treatment}
\label{rk:dry}
We note that while Theorem \ref{thm:pp} and the limiter 
\eqref{pplimit} ensures non-negativity of the water height cell average $\bar h_K^{n+1}$, and water height on the quadrature points $S_K$ at the next time level, this in general is not enough for the invertibility of the water height-weighted mass matrix, which is needed to compute the velocity approximation. 
The invertibility of this weighted mass matrix is not a big issue as one can compute the scaling factor in \eqref{scale} such that it ensures posivitity of water height on all volume integration points in each cell.

A more serious issue is the velocity computation on dry cells with nearly zero water height, this weighted mass matrix is nonsingular but now close to zero, and the computed velocity approximation may be unphysically large and not reliable anymore.
Without a special velocity treatment on dry cells, the scheme (with TVB and posivity preserving limiters) may still fail to solve challenging problems with moving interface with wet and dry areas.
There are various wetting/drying treatment available in the literature \cite{KP07, BKW09}. However, our preliminary numerical experiments suggest that the most common approaches may not work well for our velocity based DG scheme. 
For example, the simple trick of setting zero velocity when the water height $h_h$ is less than a given threshold, e.g. $10^{-6}$,  which worked in \cite{XZS10}, or using 
a regularized water height 
\[
h^*:= \frac12 h+\frac12 \max\{h, (tol)^2/h\},
\]
with $tol$ a given small tolerance, 
to compute the weighted mass matrix in Algorithm \ref{algV} as suggested 
in \cite{KP07}
were not enough for our scheme with polynomial degree $k=2$ to solve a dam break problem on a dry bed.

After some initial testing, we come up with a relative simple velocity limiting approach that works for the numerical results reported in this manuscript. We apply the following two steps after an inner stage 
water height $h_h$ and discharge $\bld m_h$ has been computed by  Algorithm \ref{alg1}:

(1)    Given a threshold percentage $0<\epsilon_{d}\ll 1$, we first mark  cells with cell average $\bar h_K \le \epsilon_d\times  h_{\max}^0$ as {\it dry} cells, where $h_{\max}^0$ is the maximum water height at initial time. Then, we remove the {\it high order} information on these dry cells by reverting to piecewise constant approximation of water height and discharge:
    \begin{align}
        \label{dry}
        h_K\leftarrow \bar h_K, \quad 
        \bld m_K\leftarrow \bar{\bld m}_K, \quad
        \text{ for all } K\in \Omega_h \text{ such that }
        \bar h_K \le \epsilon_d\times  h_{\max}^0,
    \end{align}
    where $h_K$ and $\bld m_K$ are the polynomial data in cell $K$, and
    $\bar h_K$ and $\bar {\bld m_K}$ are the cell averages. 
    Note that this approach does not affect the local conservation property, but may lead to accuracy loss. However, since there are only a small amount of water in dry cells, such loss of accuracy may not be too significant if $\epsilon_d$ is taken small enough. We note that similar treatment was used in \cite{BKW09}.
    
    (2)
    The above approach may not be enough to control the velocity magnitude for high order schemes when $\epsilon_d$ is taken to be too small. We further propose a velocity limiter to smooth out extreme velocity values.
Given a user tunable value $V_{\max}$, which is an estimation of maximal allowed velocity approximation, we do the following two steps
for each component of the velocity approximation:
\begin{itemize}
    \item [(i)] 
    Identify the collection of {\it troubled velocity cells}, denoted as $\omega(u_h)$, for the velocity
    component $u_h$ as the cells where the maximum of the absolute velocity is larger than $V_{\max}$. For polynomial degree $k=2$ on triangles, the maximal value in the triangle is estimated as the maximal value on three vertices and three mid points of each edge:
    \begin{align}
        \label{tx}
        \omega(u_h):=\{K\in \Omega_h:\quad 
        \max_{x\in v(K)}{|u_h(x)| > V_{\max}},
        \}
    \end{align}
    where $v(K)$ is the collection of three vertices and three edge midpoints of cell $K$.
    \item [(ii)] 
    On each of these troubled cells, we remove the velocity data, and  replace it by the average of cell averages of its immediate neighboring cells which are not marked as troubled cells.
    We repeat this procedure until all troubled cells have an updated (constant) velocity value:
    \begin{align}
        \label{tm}
&        \text{While $\omega(u_h)$ is not empty, do the following: }\nonumber\\
&\;\;\quad\quad        u_K\leftarrow \text{average of } \{\bar u_T\}
        \text{ for } T\in \omega(K) \text{ and } T\not \in \omega(u_h).\\
&\;\;\quad\quad \text{remove cell $K$ from $\omega(u_h)$ if its value has been updated.}\nonumber        
    \end{align}
\end{itemize}
We note that the above velocity limiting procedure does not affect the local conservation property as the water height and discharge cell averages were never changed.
The above two approaches introduce two tunable parameters, namely $\epsilon_d$ in \eqref{dry}, and $V_{\max}$ in \eqref{tx}.
They will be chosen accordingly for specific wetting/drying examples.
For example, we can take $\epsilon_d = 5\times 10^{-3}$, and take $V_{\max}$ based on the maximum velocity magnitude for the lowest order scheme with $k=0$ for problems with moving wet/dry interfaces. With the above wetting/drying treatment, we are able to run simulation for the circular dam break problem with a dry bed, and 
the dam bream problem with three mounds on unstructured triangular grids. 
We mention that the above treatments are far away from perfect yet, as they need parameter tuning, and may lead to accuracy loss near dry cells.
They only serves as initial approaches for a successful simulation of SWEs with moving wet/dry interfaces. We are planning to further investigate more 
robust and accurate wetting and drying treatments for our velocity based DG scheme in the near future.


For completeness, we list the final form of the fully discrete scheme below.
This method is locally conservative, well-balanced, and positivity preserving provided the time step size $\Delta t$ satisfy the CFL
condition \eqref{cfl}.
In practice, we take the time step size to be 
\begin{align}
\Delta t = cfl \min_{K\in\Oh}\{\tau_K/\alpha_K^{\max}\},
\end{align}
where $cfl$ is the CFL number which depends on the polynomial degree $k$, 
$\tau_K$ is the mesh size, and $\alpha_K^{\max}$ is the estimated maximum speed on the cell $K$.
If we detect a water height cell average $\bar h_K< \epsilon=10^{-12}$ in the inner stages in Step 1/2/4 of Algorithm \ref{alg3pt}, which means the time step size does not satisfy the condition \eqref{cfl}, we simply decrease $\Delta t$ by a half and redo the whole computation.

\begin{algorithm}
\caption{Posivity-preserving SSP-RK3 + DG + TVB limiter + wetting/drying treatment}
\begin{algorithmic}[1]
\REQUIRE $h_h^{n}\in V_h^k$, $\bld u_h^{n}, \bld m_h^{n}\in \bld V_h^k$ at time level $t^n$, and time step size $\Delta t>0$.
$tol >0$ for TVB limiter indicator \eqref{indicator}, 
$\epsilon_d>0$ for dry cell indicator \eqref{dry},
and 
$V_{\max}>0$ for troubled velocity cell indicator \eqref{tx},

\ENSURE $h_h^{n+1}\in V_h^k$, and $\bld u_h^{n+1}, \bld m_h^{n+1}\in \bld V_h^k$ at next time level $t^{n+1}:=t^n+\Delta t$.
  \STATE  Apply Algorithm \ref{alg1} with inputs $h_h^n, \bld u_h^n, \bld m_h^n$ and $\Delta t$.
  (If bottom topography $b_h$ is discontinuous, apply the hydrostatic reconstruction \eqref{hydro} in flux evaluations.)
  Denote the outputs as $h_h^{(1)}, \bld m_h^{(1)}$.
  \STATE Apply the dry cell limiter \eqref{dry}
  for $h_h^{(1)}$ and $\bld m_h^{(1)}$; 
  Apply the characteristic-wise TVB limiter for the variables $(h_h^{(1)}+b_h, \bld m_h^{(1)})$ using indicator \eqref{indicator}
  with indicating function
 $h_h^{(1)}+b_h$; 
    Apply the positivity preserving limiter for $h_h^{(1)}$. 
\STATE  Apply Algorithm \ref{algV} with inputs $h_h^{(1)}, \bld m_h^{(1)}$.
Denote the velocity output as $\bld u_h^{(1)}$. 
Then apply the velocity limiter in \eqref{tm}. 
\STATE Apply Algorithm \ref{alg1} with inputs $h_h^{(1)}, \bld u_h^{(1)}, \bld m_h^{(1)}$ and $\Delta t$. Denote outputs as $h_h^{(2*)}, \bld m_h^{(2*)}$.
  \STATE 
  Apply Algorithm \ref{alg2} with inputs $h_h^{n}, h_h^{(2*)}, \bld m_h^{n}, \bld m_h^{(2*)}$ and weights $w_1=0.75, w_2 = 0.25$.
  Denote outputs as  $h_h^{(2)}, \bld m_h^{(2)}$.
  \STATE Apply the dry cell limiter \eqref{dry}
  for $h_h^{(2)}$ and $\bld m_h^{(2)}$; 
  Apply the characteristic-wise TVB limiter for the variables $(h_h^{(2)}+b_h, \bld m_h^{(2)})$ using indicator \eqref{indicator}
  with indicating function
 $h_h^{(2)}+b_h$; 
    Apply the positivity preserving limiter for $h_h^{(2)}$. 
 \STATE  Apply Algorithm \ref{algV} with inputs $h_h^{(2)}, \bld m_h^{(2)}$.
Denote the velocity output as $\bld u_h^{(2)}$. 
Then apply the velocity limiter in \eqref{tm}. 
\STATE Apply Algorithm \ref{alg1} with inputs $h_h^{(2)}, \bld u_h^{(2)}, \bld m_h^{(2)}$ and $\Delta t$. Denote outputs as $h_h^{(3*)}, \bld m_h^{(3*)}$.  
 \STATE 
  Apply Algorithm \ref{alg2} with inputs $h_h^{n}, h_h^{(3*)}, \bld m_h^{n}, \bld m_h^{(3*)}$ and weights $w_1=1/3, w_2 = 2/3$.
  Denote outputs as  $h_h^{n+1}, \bld m_h^{n+1}$.
  \STATE Apply the dry cell limiter \eqref{dry}
  for $h_h^{n+1}$ and $\bld m_h^{n+1}$; 
  Apply the characteristic-wise TVB limiter for the variables $(h_h^{n+1}+b_h, \bld m_h^{n+1})$ using indicator \eqref{indicator}
  with indicating function
 $h_h^{n+1}+b_h$; 
    Apply the positivity preserving limiter for $h_h^{n+1}$. 
   \STATE  Apply Algorithm \ref{algV} with inputs $h_h^{n+1}, \bld m_h^{n+1}$.
Denote the velocity output as $\bld u_h^{n+1}$. 
Then apply the velocity limiter in \eqref{tm}. 
\end{algorithmic}
\label{alg3pt}
\end{algorithm}

\section{Numerical results}
In this section we present numerical results of our velocity based DG scheme Algorithm \eqref{alg3pt}. We report results using the third order DG method with $k = 2$. The CFL number is taken to be $cfl = 0.1$ for 1D examples, and $cfl=0.05$ for 2D examples.
The gravitation constant $g$ is fixed as 9.812 except the test in Example 4.6, where it is taken to be $g=10$.
We take the tolerance $tol = 0.02$ in the Fu-Shu indicator \eqref{indicator} for all examples. 
Moreover, unless explicitly mentioned, we turn off the dry cell limiter \eqref{dry} and the velocity limiter \eqref{tx}--\eqref{tm}, which are only needed when the problem has a moving dry/wet interface.

The implementation is based on the python interface of the NGSolve 
software \cite{Schoberl16}, \url{https://ngsolve.org/}.
Source code for all the examples can be found in the git repository, \url{https://github.com/gridfunction/SWE}.

\subsection*{Example 4.1: Accuracy Test in 1D}
We start with an accuracy test to demonstrate the high order accuracy of our schemes for a smooth solution of the SWEs. Following the setup in \cite{WDGX20}, we take the following bottom topography and initial
conditions:
\begin{align*}
    &b(x) =\;  \sin^2(\pi x),\quad h(x,0) = \; 5 + e^{\cos(2\pi x)},\quad hu(x,0) = \;\sin(\cos(2\pi x)).
\end{align*}
The computation domain is a periodic unit interval $[0, 1]$, and final time is $t=0.1$ where the solution is still smooth.
We apply the plain Algorithm \ref{alg3} without limiter, and compute the $L^2$-errors of water height $h_h$, velocity $u_h$, and discharge $m_h$ on a sequence of uniform meshes with $N=50\times 2^l$ cells for $l=0,1,2,3,4$. We take the solution on $N=25\times 2^5=1600$ cells 
as the {\it reference}  solution when computing these $L^2$-errors.
The results are recorded in Table \ref{tab:1}.
We clearly observe the expected third order convergence, and the error magnitude are comparable to the results reported in \cite{WDGX20}
for another third order entropy stable DG scheme.
\begin{table}[ht!]
    \centering
    \begin{tabular}{c|cc|cc|cc}
    N    & $L^2$-err in $h_h$ & rate 
    & $L^2$-err in $u_h$ & rate
&     $L^2$-err in $m_h$ & rate\\
\hline
 50& 2.997e-04 & -- & 3.583e-04 & -- & 2.577e-03 &-- \\
100& 2.730e-05 & 3.46 & 3.273e-05 & 3.45 & 2.352e-04 & 3.45 \\
200& 2.949e-06 & 3.21 & 3.538e-06 & 3.21 & 2.542e-05 & 3.21 \\
400& 3.600e-07 & 3.03 & 4.323e-07 & 3.03 & 3.103e-06 & 3.03 \\
800& 4.408e-08 & 3.03 & 5.296e-08 & 3.03 & 3.798e-07 & 3.03 \\
\hline
    \end{tabular}
    \vspace{1ex}
    \caption{$L^2$ errors and
convergence rate at time $t=0.1$ for  Example 4.1.}
    \label{tab:1}
\end{table}

\subsection*{Example 4.2: The Well-Balanced Test in 1D}
In this example, we test the well-balanced property of our proposed methods to ensure that
the still-water steady state is exactly preserved. 
We consider two different choices of the bottom topography as used in \cite{WDGX20}:
a smooth bottom with 
\[
b(x) = 5\exp\left(-0.4(x-5)^2
\right),
\]
and a discontinuous bottom with 
\[
b(x) = \left\{
\begin{tabular}{ll}
4, & if $4\le x\le 8$,\\[.3ex]
0, & otherwise.
\end{tabular}
\right.
\]
The computational domain is $[0, 10]$ with wall boundary conditions.
The initial condition is taken as the stationary state
\[
h+b = 10, u = 0.
\]
We solve the problem until time $t=0.5$ on three meshes with 100, 2000, and 400 uniform cells, and record the $L^2$-errors in Table \ref{tab:2}. We observe all errors are at the level of
round-off errors, which verifies the well-balanced property.
\begin{table}[ht!]
    \centering
    \begin{tabular}{c|ccc|ccc}
    & & smooth bot. & && disc. bot. & 
    \\
    \hline
    N    & $L^2$-err in $h_h$ 
    & $L^2$-err in $u_h$ 
&     $L^2$-err in $m_h$ 
& $L^2$-err in $h_h$ 
    & $L^2$-err in $u_h$ 
&     $L^2$-err in $m_h$ \\
\hline
100& 9.819e-14 & 7.380e-14 & 5.025e-13 & 7.700e-14 & 7.627e-14 & 5.024e-13 \\
200& 1.747e-13 & 8.344e-14 & 6.331e-13 & 1.805e-13 & 1.042e-13 & 8.494e-13 \\
400& 3.740e-13 & 1.328e-13 & 1.029e-12 & 2.787e-13 & 2.153e-13 & 1.355e-12 \\
\hline
    \end{tabular}
    \vspace{1ex}
    \caption{$L^2$ errors at time $t=0.5$ for  Example 4.2.}
    \label{tab:2}
\end{table}

\subsection*{Example 4.3: A Small Perturbation Test in 1D}
We test the following quasi-stationary test case proposed by LeVeque \cite{LeVeque98}, which is a small
perturbation of the steady state solution.
The computational domain is $[0, 2]$, and the bottom topography b(x) is given by
\[
b(x) = \left\{
\begin{tabular}{ll}
$\frac14(\cos(10\pi(x-1.5))+1)$, & if $1.4\le x\le 1.6$,\\[.7ex]
0, & otherwise.
\end{tabular}
\right.
\]
The initial conditions are 
\[
h(x,0) = \left\{
\begin{tabular}{ll}
$1-b(x)+\epsilon$, & if $1.1\le x\le 1.2$,\\[.7ex]
$1-b(x)$, & otherwise.
\end{tabular}
\right., \quad u(x,0) = 0,
\]
where $\epsilon$ is a given constant representing the size of the perturbation.
Following \cite{LeVeque98},
we consider  a
 case with a big pulse ($\epsilon = 0.2$) and case with a small pulse ($\epsilon=0.001$). The final time of simulation is $t=0.2$.
We compare our scheme on a uniform coarse mesh with $N=200$ cells and 
a uniform fine mesh with $N=2000$ cells. 
The results at final time 
for  the total water surface $h+b$ and discharge $m=hu$ for the big pulse case
are shown in Figure \ref{fig:sp1}, and 
those for the small pulse case
are shown in Figure \ref{fig:sp2}.
In these figures, blue squares indicate the troubled cells identified by our indicator \eqref{indicator}.
We observe good agreement of the results on coarse and fine meshes without spurious
numerical oscillations, which also agrees well with results in the literature. 
Moreover, we observe that the indicator \eqref{indicator} with $tol=0.02$ only activates cells close to the moving shocks for both case, with slightly more cells being identified as troubled cells for the small perturbation test. 

\begin{figure}[ht]
\centering
\includegraphics[width=0.45\textwidth]{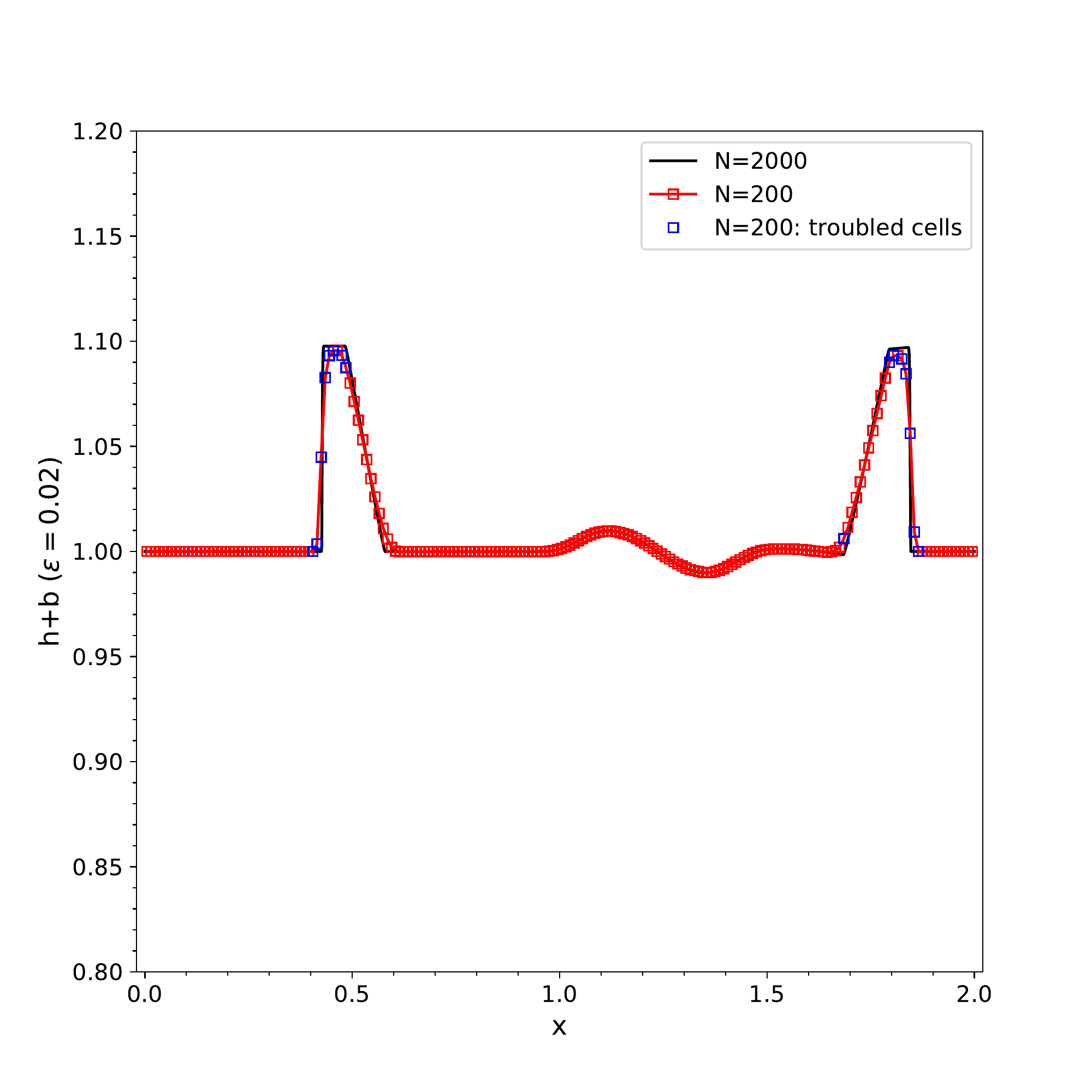}
\includegraphics[width=0.45\textwidth]{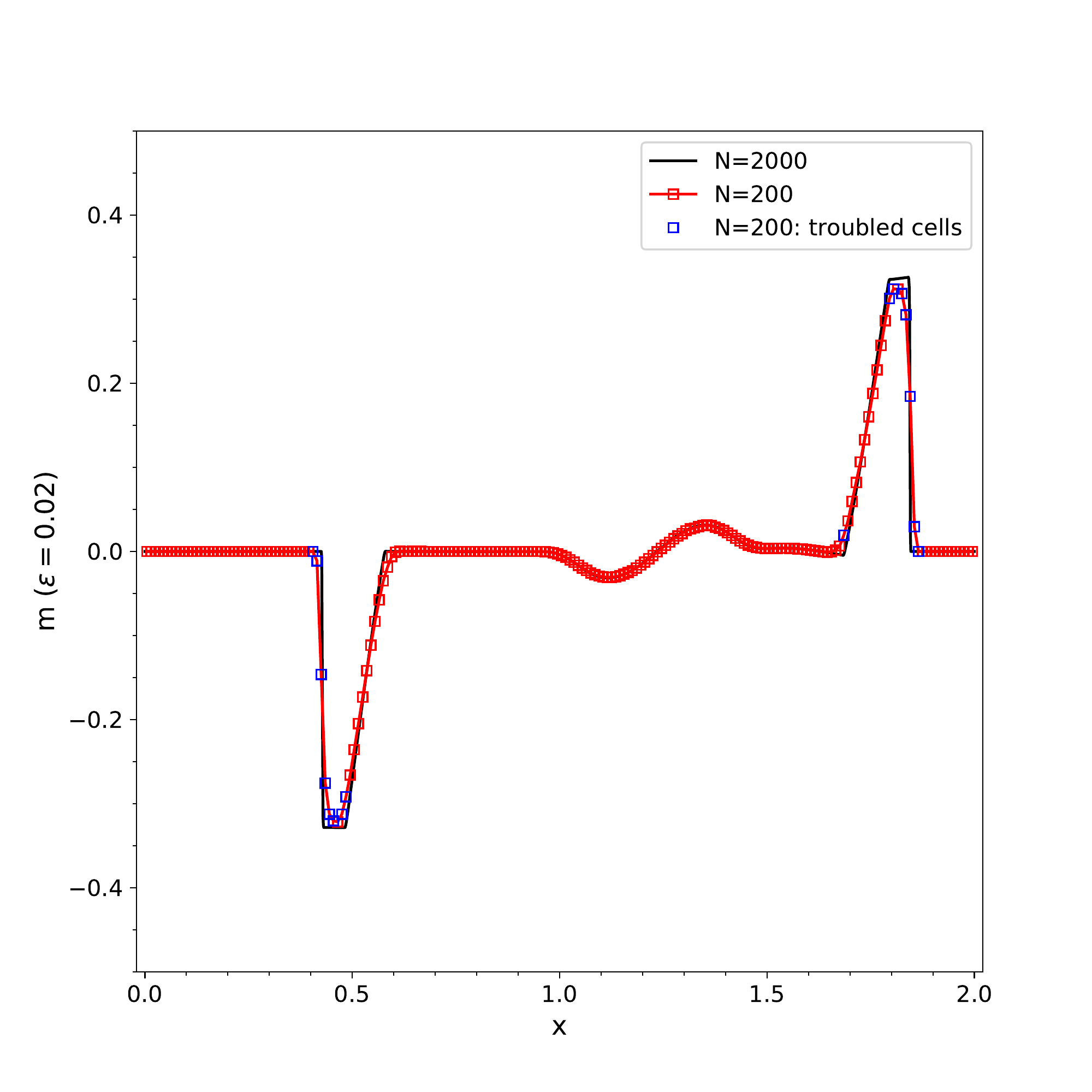}
\caption{Example 4.3 with 
with a big pulse 
$\epsilon=0.2$  at time $t = 0.2$. Left: the water
surface $h + b$; right: the discharge $m$. Blue squares indicate cells where TVB limiter are used at the final time.
}
\label{fig:sp1}
\end{figure}

\begin{figure}[ht]
\centering
\includegraphics[width=0.45\textwidth]{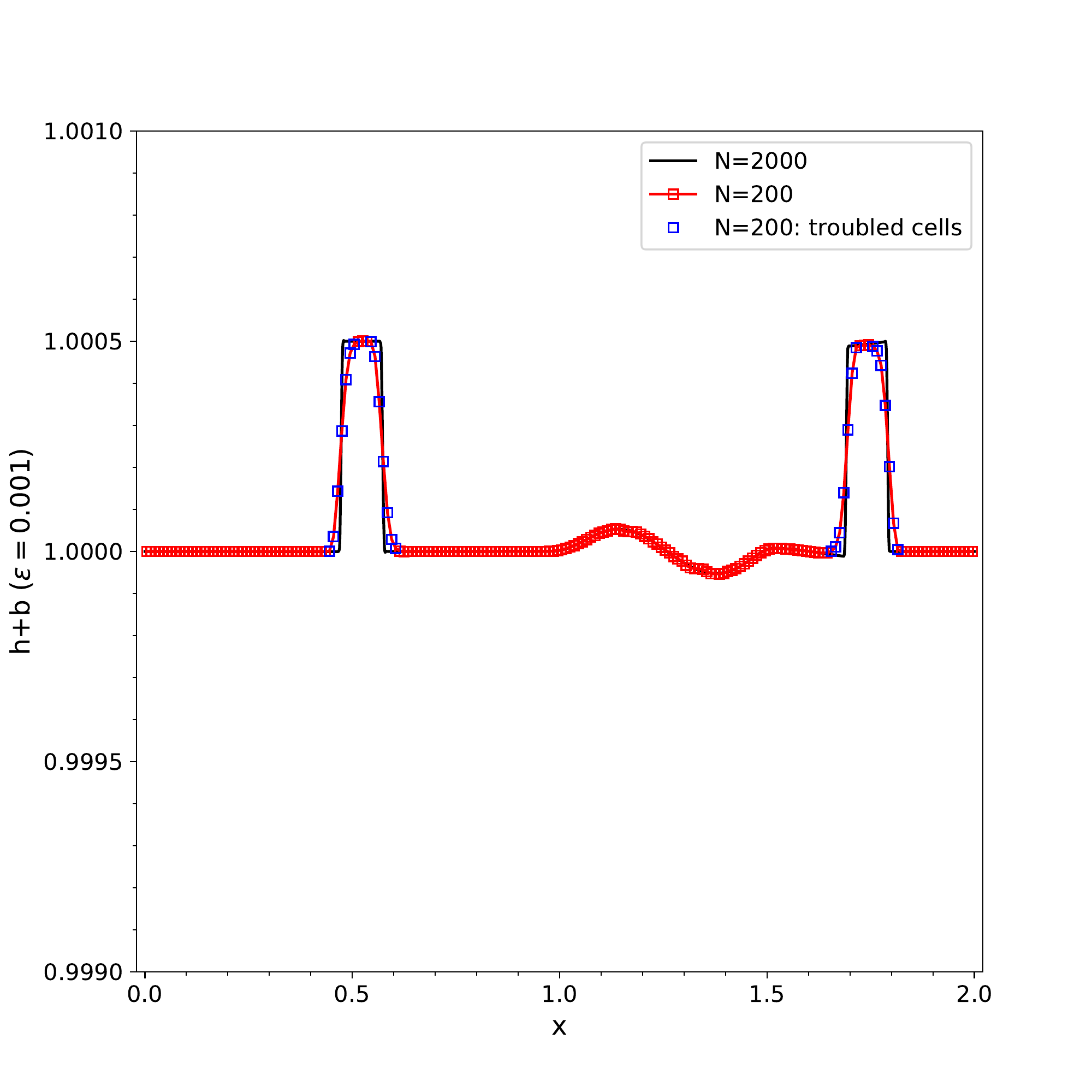}
\includegraphics[width=0.45\textwidth]{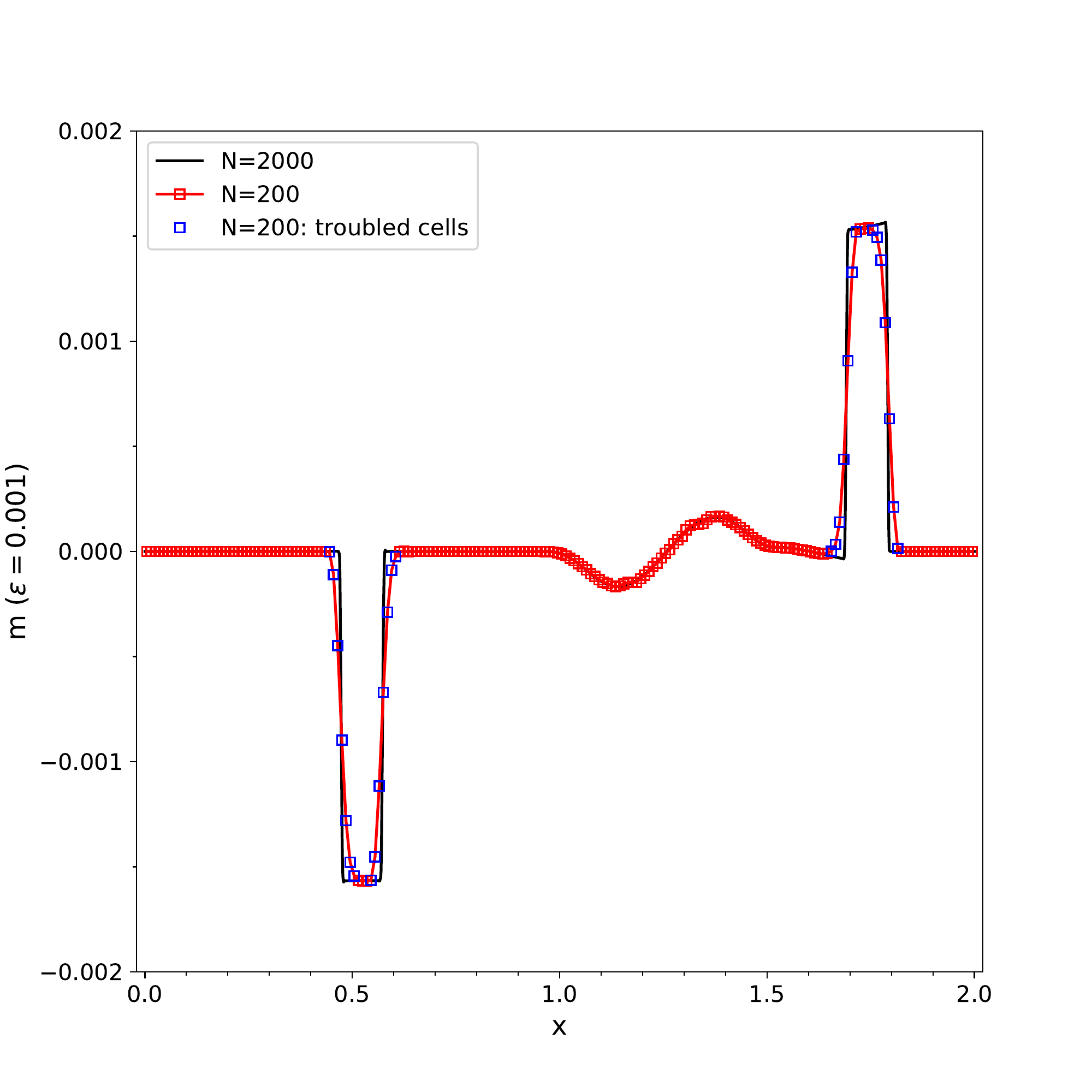}
\caption{Example 4.3 with 
with a small pulse 
$\epsilon=0.001$  at time $t = 0.2$. Left: the water
surface $h + b$; right: the discharge $m$. Blue squares indicate cells where TVB limiter are used at the final time.
}
\label{fig:sp2}
\end{figure}

\subsection*{Example 4.4: A Dam Breaking Problem over a Bump in 1D}
We consider a one-dimensional dam breaking problem over a rectangular bump.
It involves a rapidly varying flow over a discontinuous bottom topography.
Following \cite{WDGX20}, we take the computational domain as [0, 1500], and use the following discontinuous bottom topography:
\[
b(x) = \left\{
\begin{tabular}{ll}
8, & if $|x-750|<1800/8$,\\[.7ex]
0, & otherwise.
\end{tabular}
\right.
\]
We use outflow boundary conditions, and record the results at time $t=60$ in Figure \ref{fig:db}, again using a uniform coarse mesh with $N=200$ cells, and a uniform fine mesh with $N=2000$ cells.  
We observe good agreement of the results on two meshes, which also agrees well with results reported in the literature.  We note that the discharge has a small kink near 
$x=937.5$, where the discontinuous of bottom topography happens. Also, the indicator is successful in identifying solution discontinuities.
\begin{figure}[ht]
\centering
\includegraphics[width=0.45\textwidth]{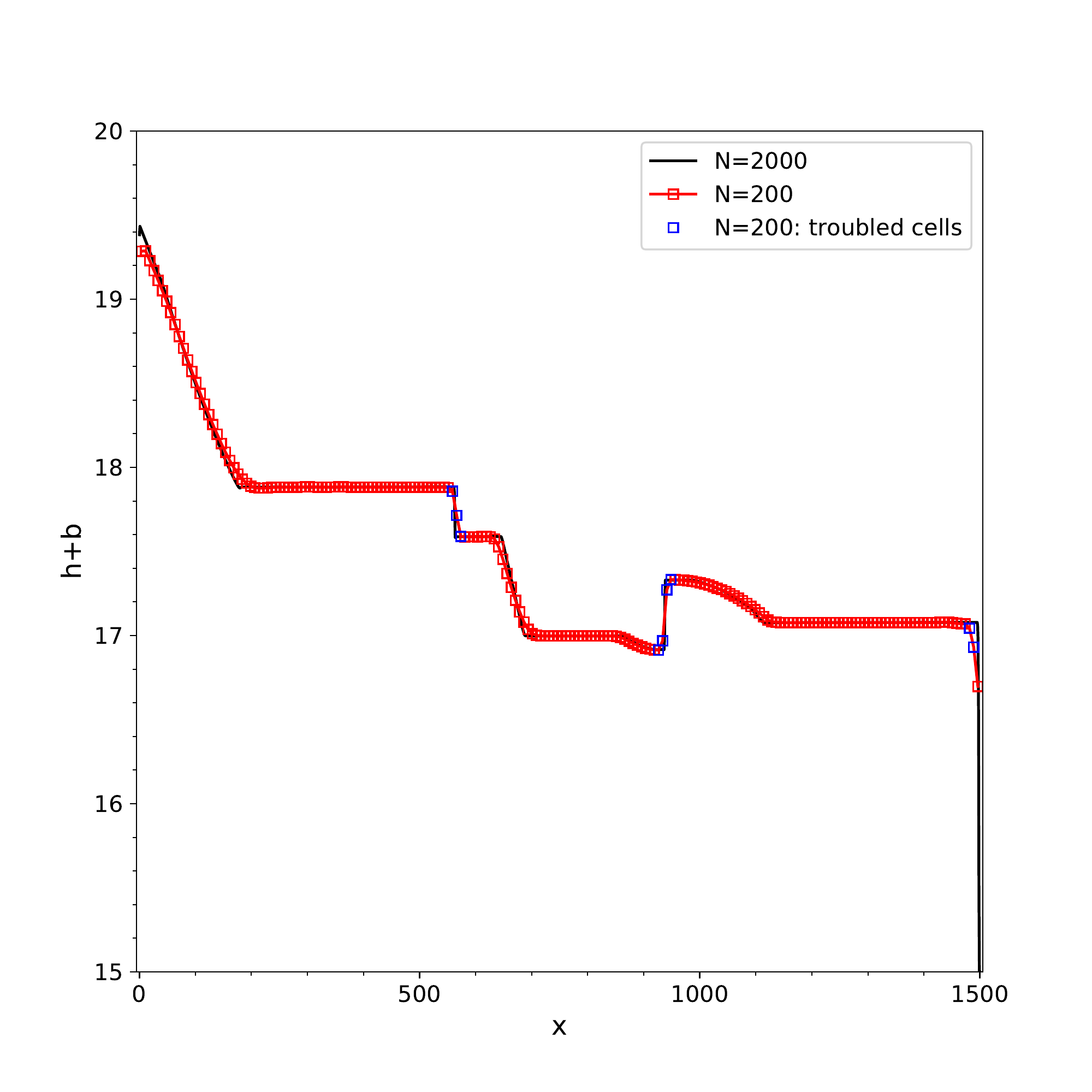}
\includegraphics[width=0.45\textwidth]{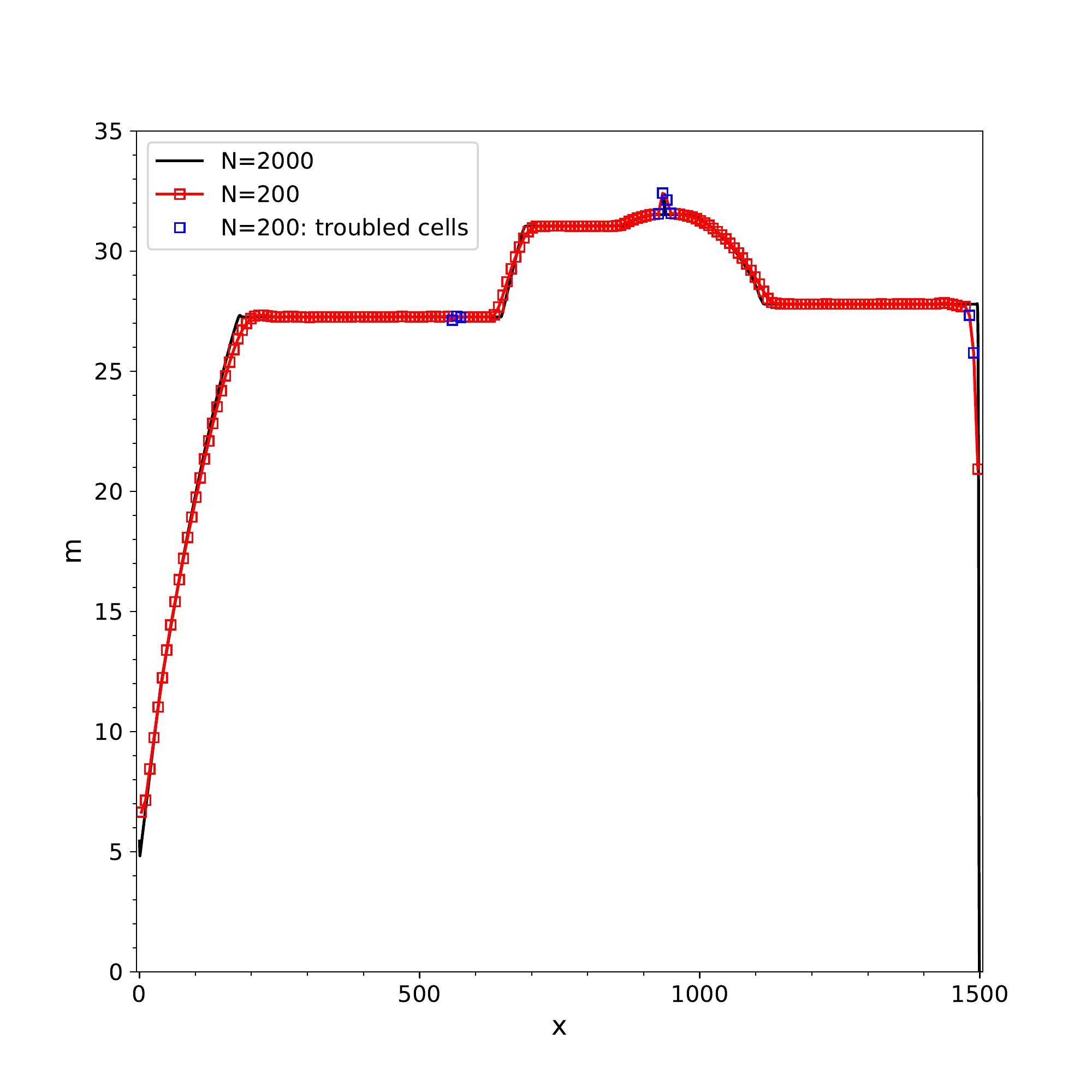}
\caption{Example 4.4 at time $t = 60$. Left: the water
surface $h + b$; right: the discharge $m$. Blue squares indicate cells where TVB limiter are used at the final time.
}
\label{fig:db}
\end{figure}

\subsection*{Example 4.5: Entropy Glitch Test in 1D}
We consider the Riemann problem with a flat bottom considered in \cite{WWAGW18}. 
The computational domain is $[-1,1]$, the bottom topography $b(x)=0$,  and initial condition is 
\[
h(x,0) = \left\{
\begin{tabular}{ll}
1, & if $x<0$,\\[.7ex]
0.1, & otherwise.
\end{tabular}
\right.\quad u(x,0) = 0.
\]
The gravitational constant is taken to be $g=10$, and final time of simulation is 
$t=0.2$.

It was shown in \cite[Fig. 8]{WWAGW18} that standard DG method with a local Lax–Friedrichs numerical flux develops an unphysical discontinuity, called an “entropy glitch”, at $x = 0$, while the entropy stable DG method is able
to capture the solution well on the coarse mesh.
The results in a uniform mesh with $200$ cells are shown in Figure \ref{fig:gt}.
\begin{figure}[ht]
\centering
\includegraphics[width=0.45\textwidth]{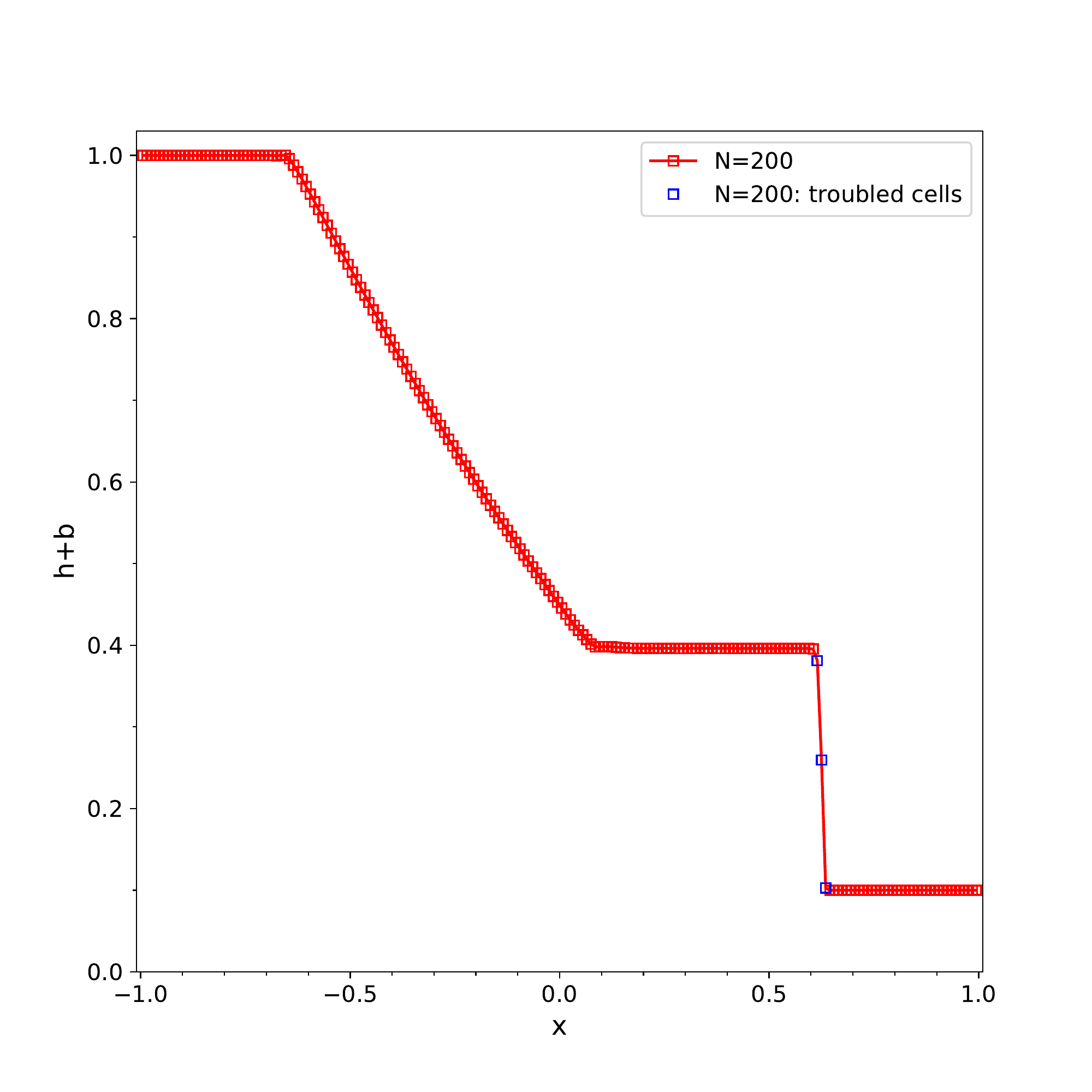}
\includegraphics[width=0.45\textwidth]{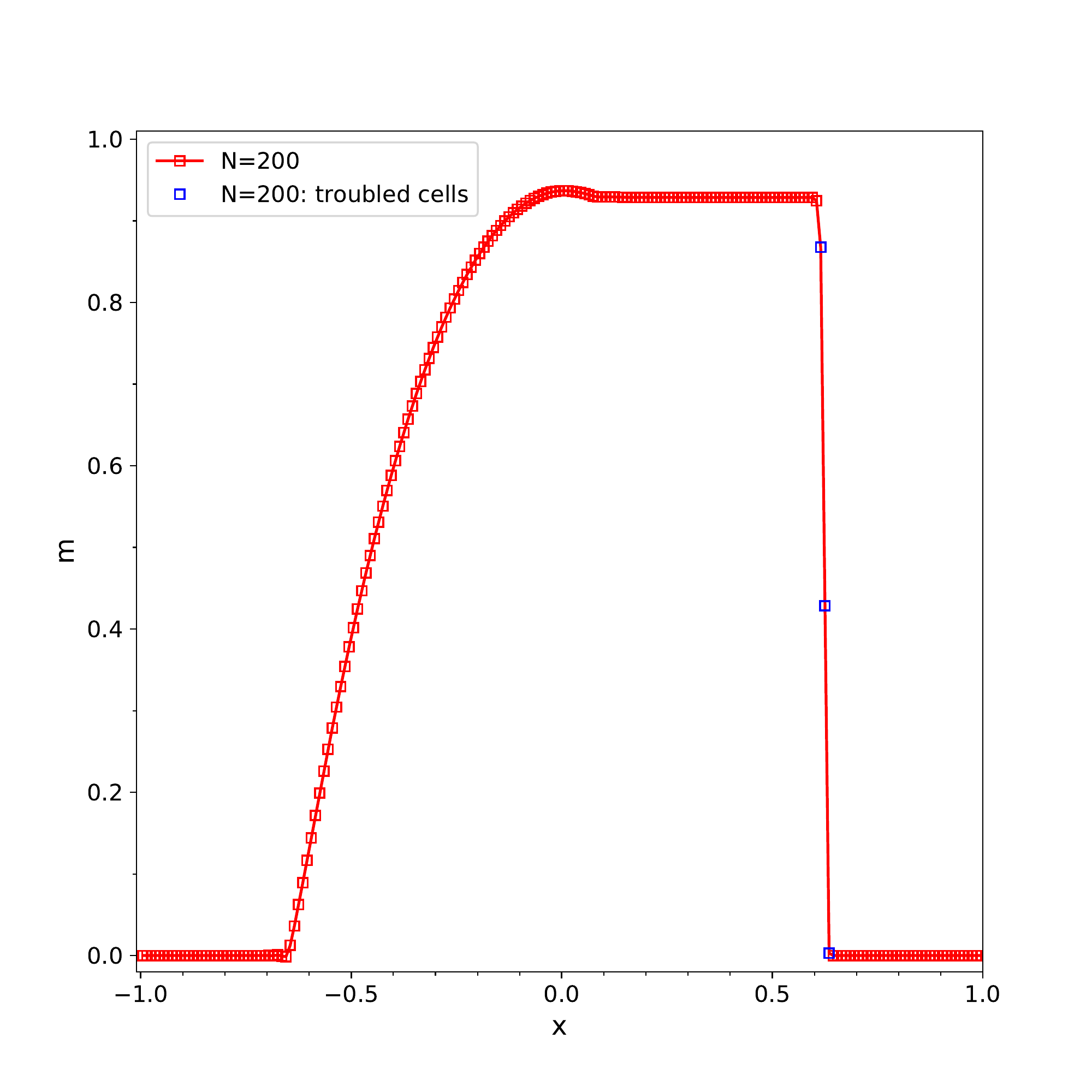}
\caption{Example 4.5 at time $t=0.2$.  
Left: the water
surface $h + b$; right: the discharge $m$.
Blue squares indicate cells where TVB limiter are used at the final time.
}
\label{fig:gt}
\end{figure}

\subsection*{Example 4.6: A Dam Break Problem with a Dry Bed in 1D}
We consider a Riemann Problem with a constant bottom used in \cite{XZS10}. 
Here a dried river bed is used to examine the performance of our scheme in case of moving wet/dry interface. The computation domain is taken to be $[-300, 300]$, and initial condition is 
\[
h(x,0) = \left\{
\begin{tabular}{ll}
10, & if $x<0$,\\[.7ex]
$10^{-12}$, & otherwise.
\end{tabular}
\right.\quad u(x,0) = 0.
\]
Here we use a tiny positive value $10^{-12}$ to indicate the dry bed.
This avoids division by zero in the velocity computation.
This is a very challenging problem as our default algorithm without dry cell limiter or 
velocity limiter fails after a couple of time steps due to an excessive large velocity approximation. 
Here we activate the dry cell limiter \eqref{dry} with $\epsilon_d=5\times 10^{-3}$ to avoid excessive large velocity approximations. The velocity limiter  \eqref{tx}--\eqref{tm} is not needed for this example.
The results at times $t=4$, $t=8$, and $t=12$ on the uniform mesh with $N=200$ cells are plotted in 
Figure \ref{alg3}.
Again, we observe good agreement with results in the literature.
\begin{figure}[ht]
\centering
\includegraphics[width=0.45\textwidth]{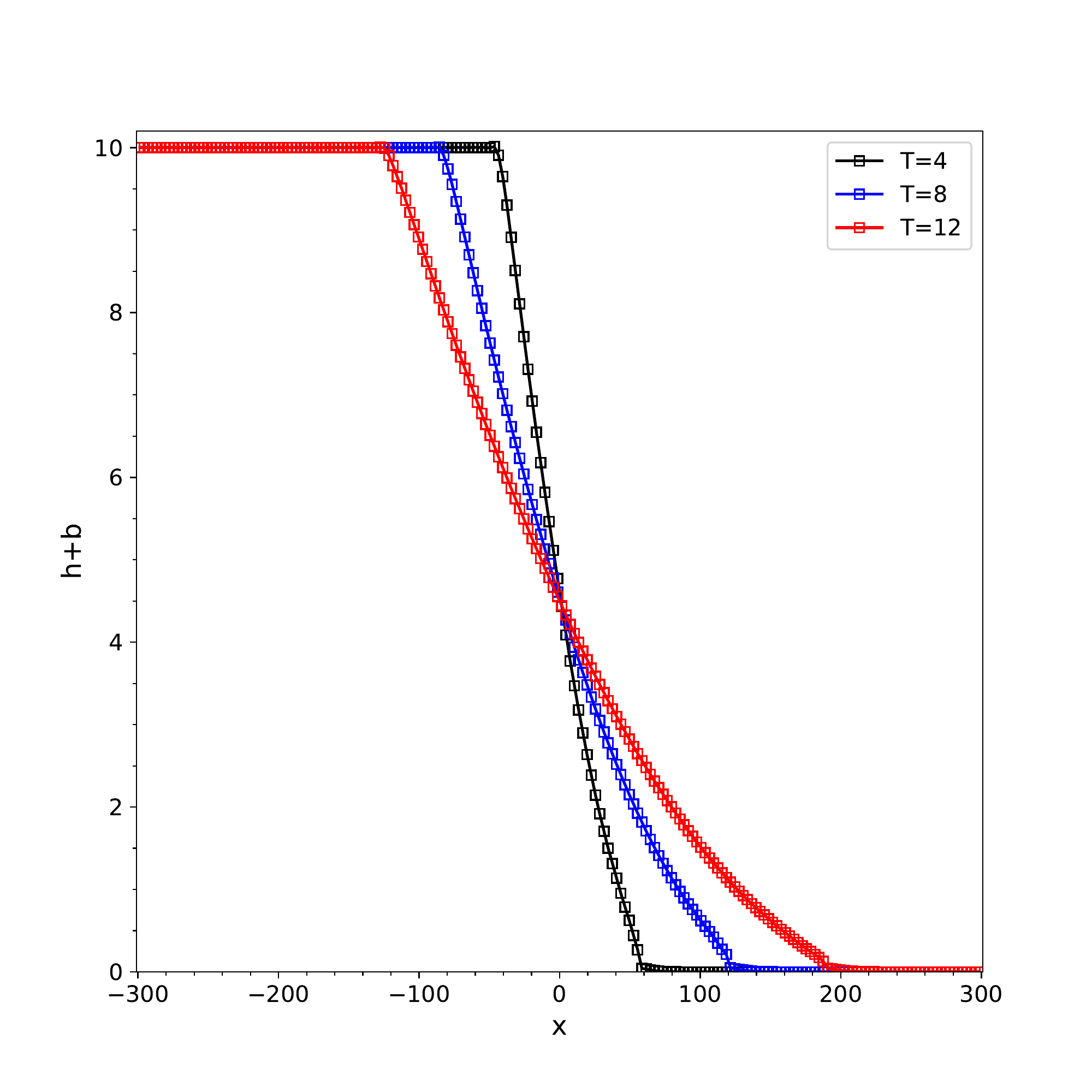}
\includegraphics[width=0.45\textwidth]{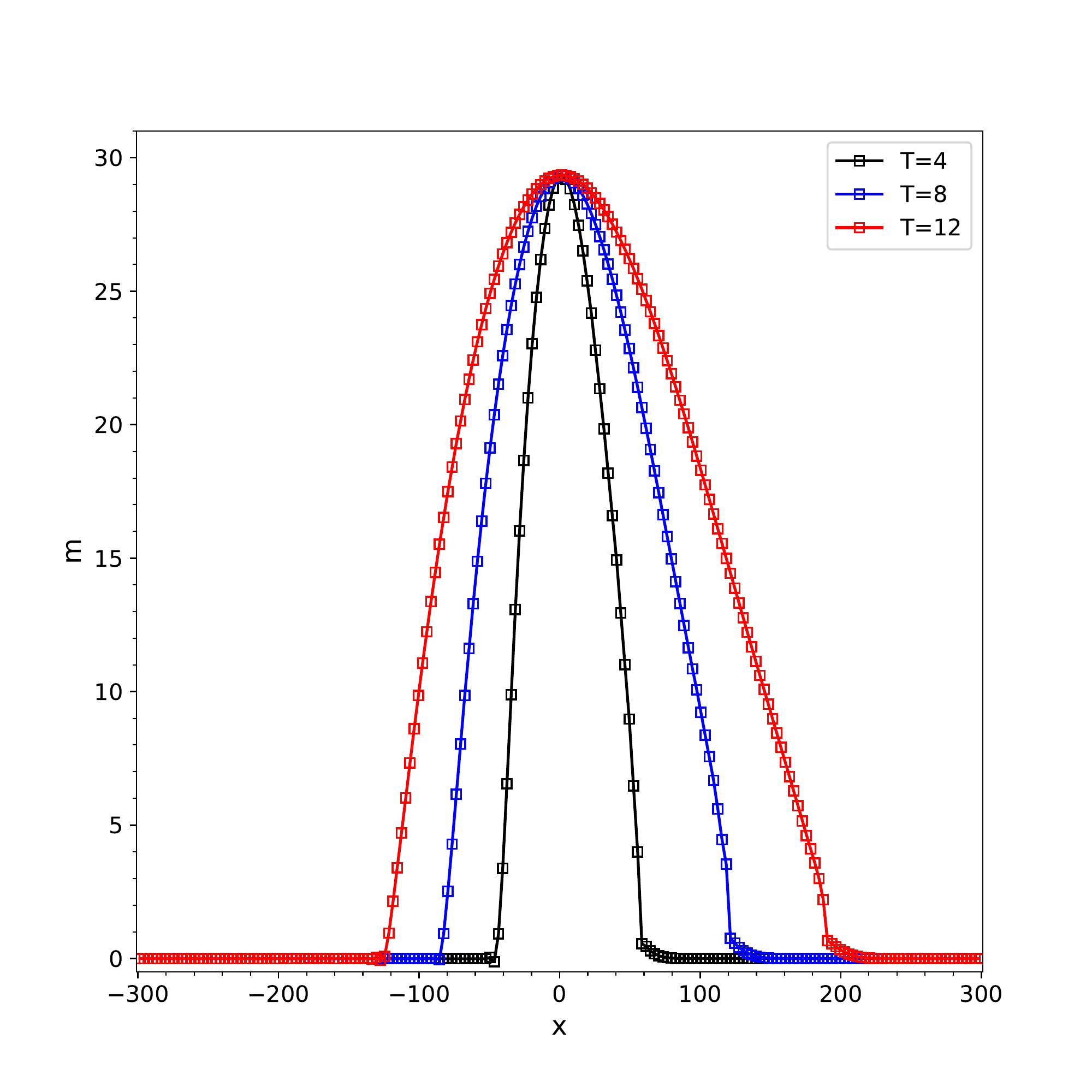}
\caption{Example 4.6 at times $t=4$, $t=8$, and $t=12$.  
Left: the water
surface $h + b$; right: the discharge $m$.
}
\label{fig:pp}
\end{figure}

\subsection*{Example 4.7: Accuracy Test in 2D}
Now we turn to the performance of our scheme on 2D  triangular meshes.
This is a 2D version of the 1D accuracy test considered in Example 4.1.
The domain is a periodic unit square $[0,1]\times [0, 1]$.
The bottom topography and the initial conditions are given as follows:
\begin{align*}
    &b(x) =\;  \sin(2\pi x)+\sin(2\pi y),\\
    & h(x,0) = \; 10 + e^{\sin(2\pi x)}\cos(2\pi y),\\
    & hu(x,0) = \;\sin(\cos(2\pi x))\sin(2\pi y),\\
    & hv(x,0) = \;\cos(2\pi x)\cos(\sin(2\pi y)).
\end{align*}
The  final time is $t=0.05$ where the solution is still smooth.
We apply the plain Algorithm \ref{alg3} without limiter, and compute the $L^2$-errors of water height $h_h$, velocity $u_h$, and discharge $m_h$ on a sequence of uniform structured triangular meshes with $N\times N\times 2$ cells where $N=25\times 2^l$ cells for $l=0,1,2,3$. We take the solution on 
uniform structured triangular meshes with $400\times 400\times 2$
as the {\it reference}  solution when computing these $L^2$-errors.
The results are recorded in Table \ref{tab:3}.
We again observe the expected third order convergence for the water height, and nearly third order convergence for the velocity and discharge.
\begin{table}[ht!]
    \centering
    \begin{tabular}{c|cc|cc|cc}
    N    & $L^2$-err in $h_h$ & rate 
    & $L^2$-err in $u_h$ & rate
&     $L^2$-err in $m_h$ & rate\\
\hline
 25& 1.420e-03 & -- & 1.670e-03 & -- & 1.379e-02 &-- \\
 50& 1.567e-04 & 3.18 & 2.083e-04 & 3.00 & 1.683e-03 & 3.03 \\
100& 1.917e-05 & 3.03 & 2.820e-05 & 2.89 & 2.371e-04 & 2.83 \\
200& 2.363e-06 & 3.02 & 3.880e-06 & 2.86 & 3.403e-05 & 2.80 \\
\hline
    \end{tabular}
    \vspace{1ex}
    \caption{$L^2$ errors and
convergence rate at time $t=0.05$ for  Example 4.7.}
    \label{tab:3}
\end{table}

\subsection*{Example 4.8: A Small Perturbation Test in 2D}
We test the following 2D quasi-stationary test case proposed by LeVeque \cite{LeVeque98}, which is a small
perturbation of the steady state solution.
The computation domain is $[0, 2] \times  [0, 1]$. The bottom topography consists of an elliptical shaped hump
\[
b(x,y) = 0.8\exp(-5(x-0.9)^2-50(x-0.5)^2),
\]
and the initial conditions are 
\[
h(x,y,0) = \left\{
\begin{tabular}{ll}
$1-b(x,y)+0.01$, & if $0.05\le x\le 0.15$,\\[.7ex]
$1-b(x,y)$, & otherwise.
\end{tabular}
\right., \quad u(x,y,0) = v(x,y,0) =0.
\]
Outflow boundary conditions are imposed on the left and right boundary while wall boundary condition (symmetry) are imposed on the top and bottom boundaries. 
Due to symmetry, we perform the calculation on half of the domain 
$\Omega = [0, 2]\times [0, 0.5]$ with symmetry boundary condition on the top boundary 
$y=0.5$.  We consider our scheme on an unstructured triangular mesh with mesh size $\tau_K = 0.01$. The water surface contour at times $t=0.12, 0.24,0.36,0.48,0.60$ are recorded in Figure \ref{fig:sp3}. Our scheme produce non oscillatory solutions and the results agrees well with those in the literature.
\begin{figure}[ht]
\centering
\includegraphics[width=0.45\textwidth]{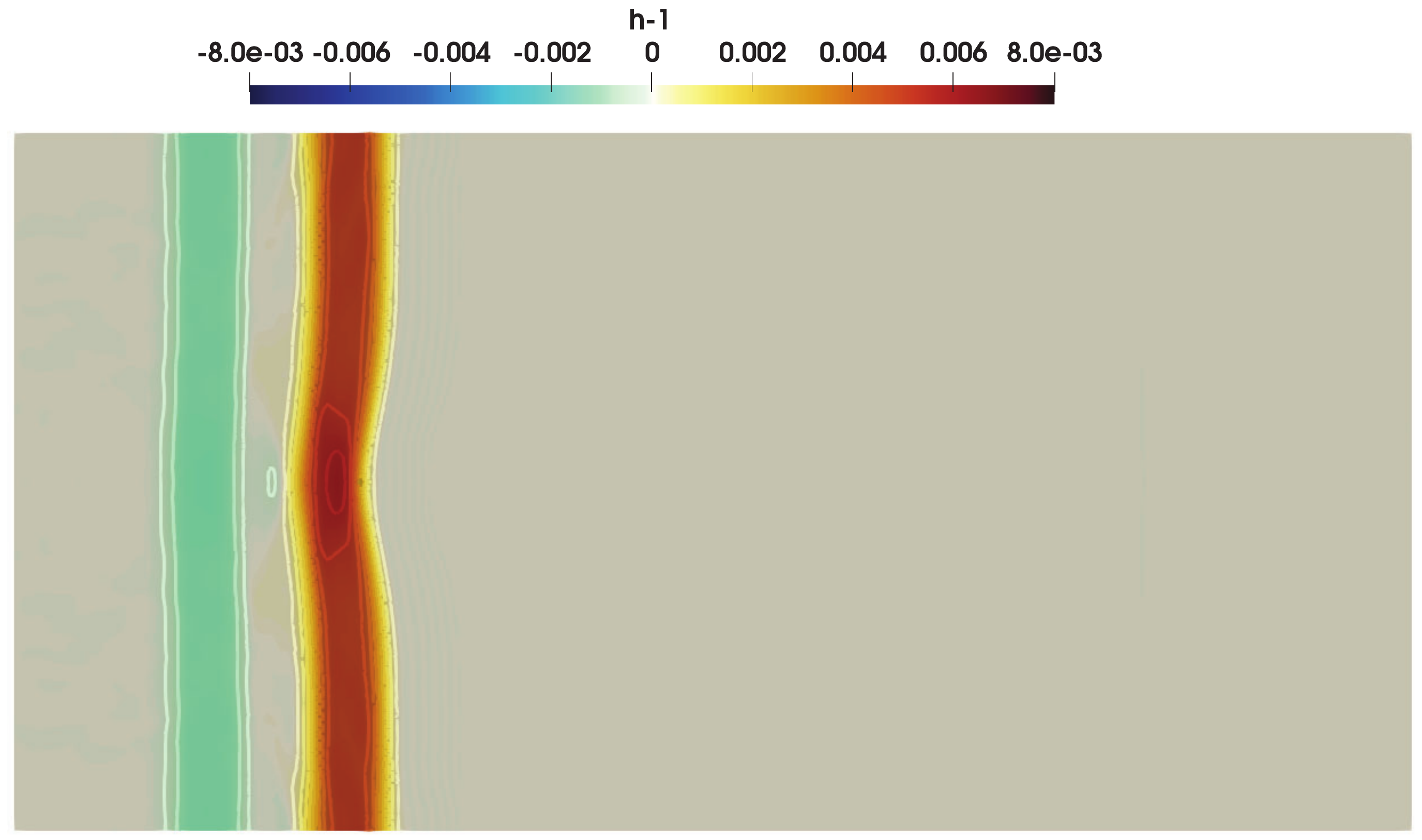}
\includegraphics[width=0.45\textwidth]{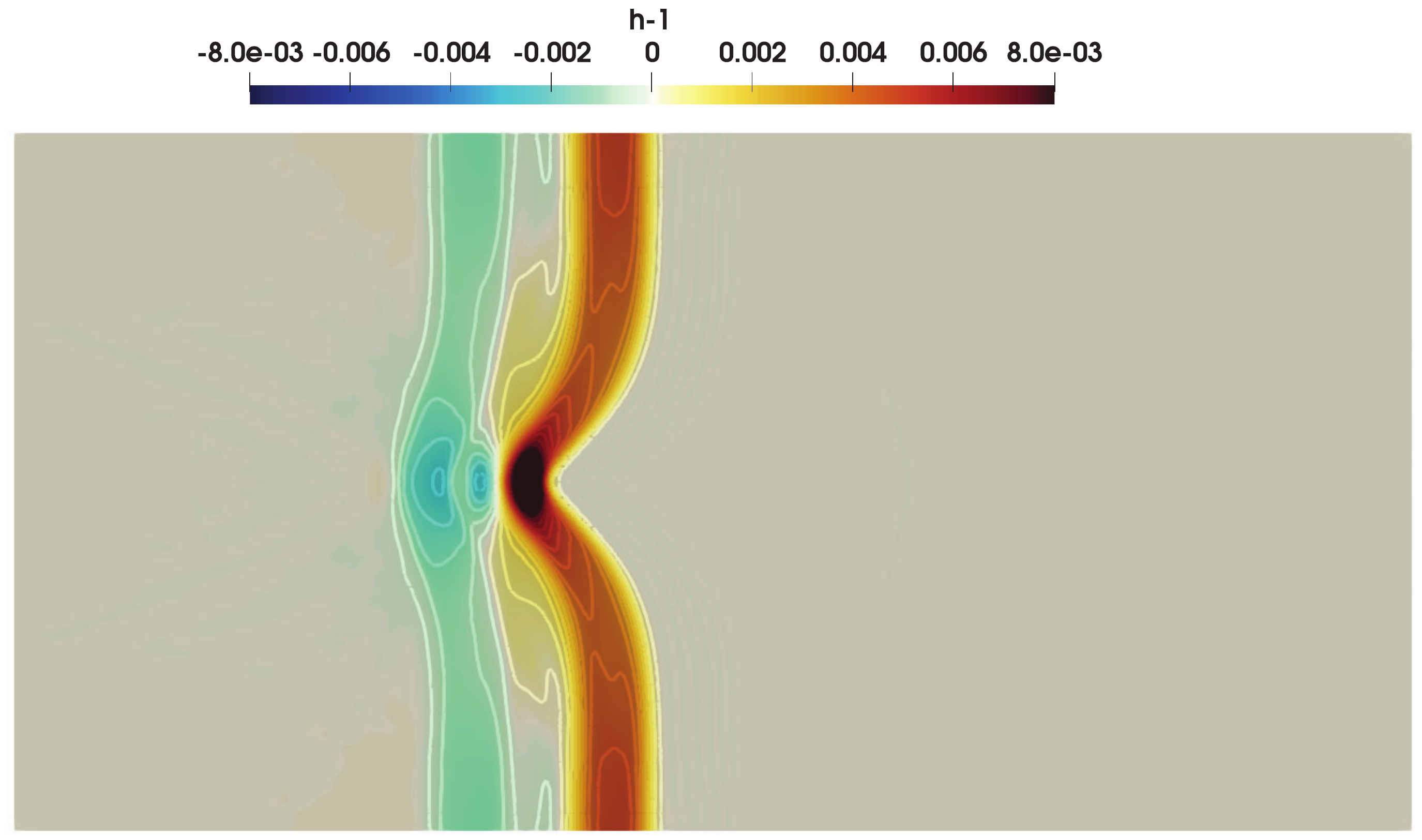}
\includegraphics[width=0.45\textwidth]{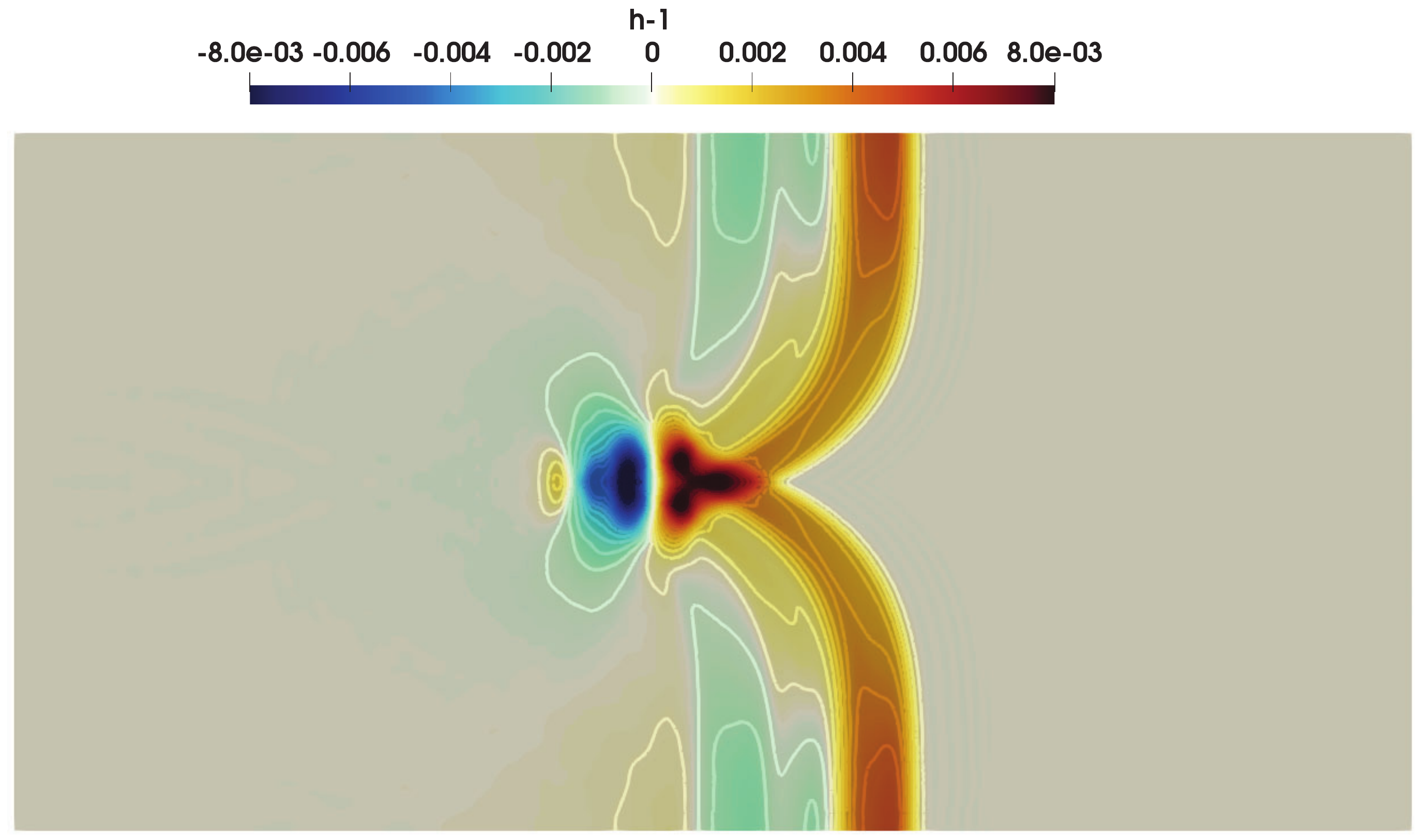}
\includegraphics[width=0.45\textwidth]{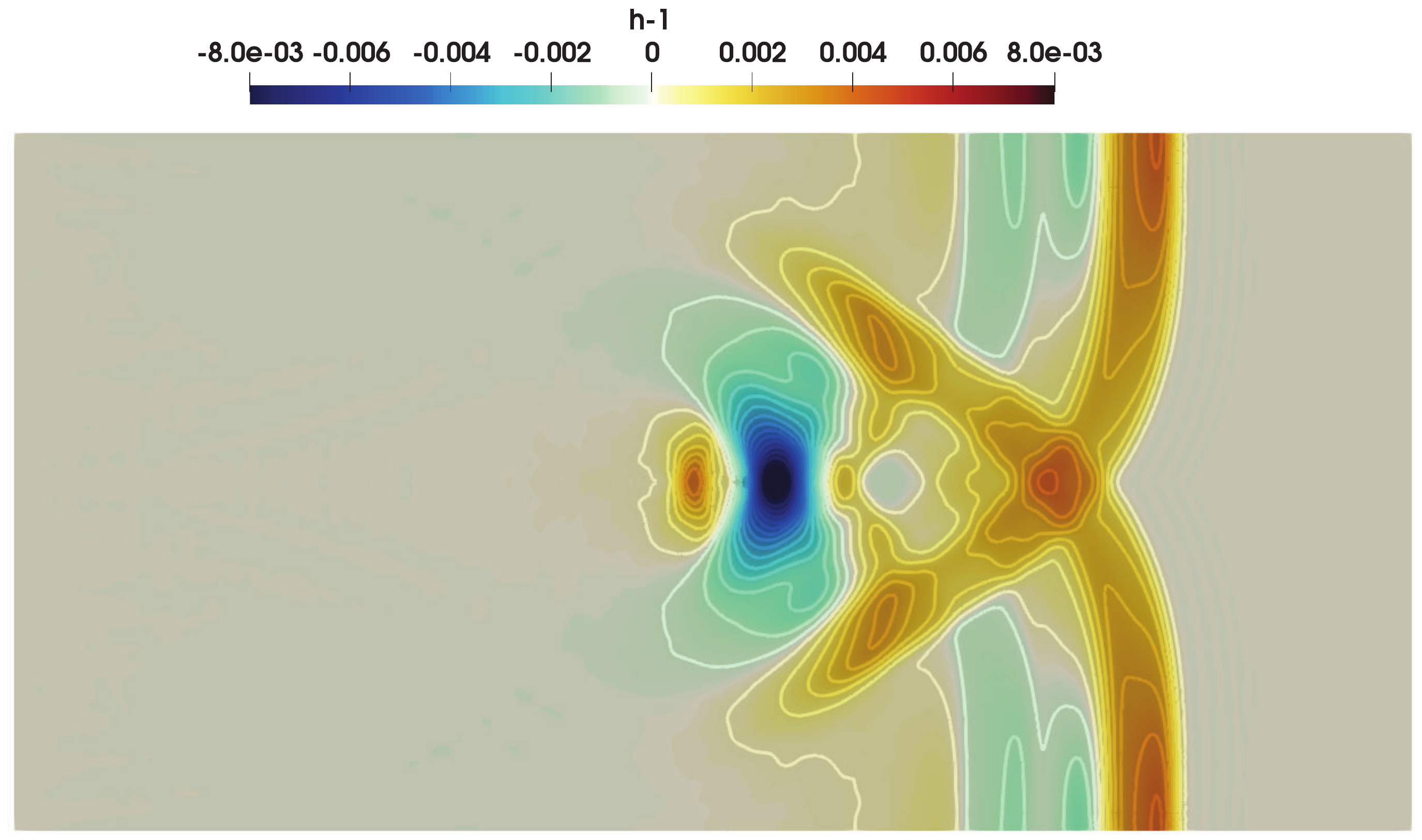}
\includegraphics[width=0.45\textwidth]{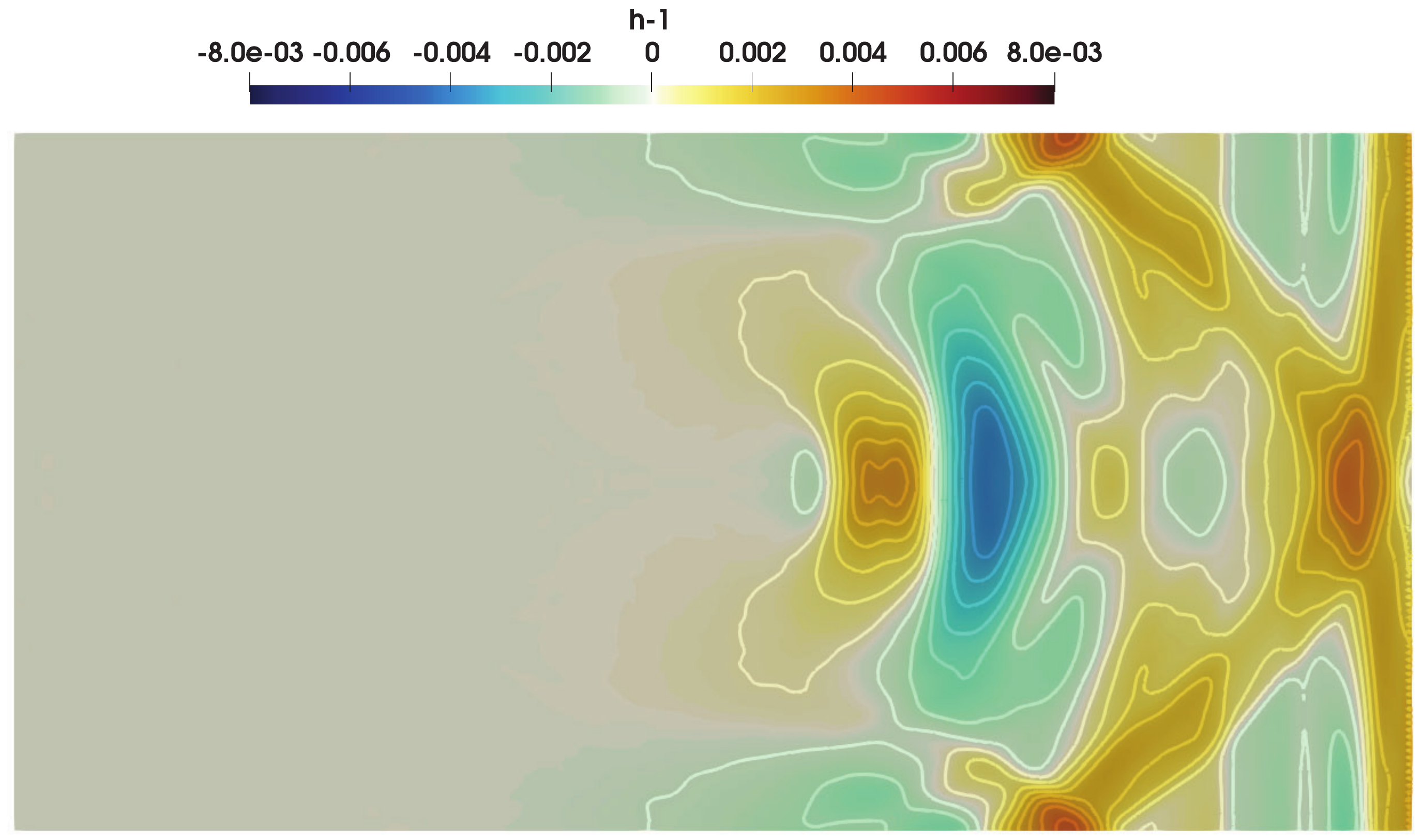}
\caption{Example 4.8.
Contour plot for $h+b-1$. 30 uniform contour lines from 
$h+b-1 = -0.008$ to 
$h+b-1 = 0.008$. 
From left to right, top to bottom: $t=0.12, 0.24,0.36,0.48,0.6$.
}
\label{fig:sp3}
\end{figure}

\subsection*{Example 4.9: Circular Dam Break Test in 2D}
We consider the circular dam break problem used in \cite[Sect. 3.2.1]{M09}.
The space domain is a $50\times 50$ square with a cylindrical dam with radius $r=11$ and centred in the square. The initial water height is 10 inside the dam, and is either 
$1$ outside the dam (a wet bed), or $10^{-12}$ outside the dam (a dry bed). The final time of the simulation is $t=0.69$.
The bottom topography is set to be zero.
Here the dry bed case need special care, where we applied the dry cell limiter \eqref{dry} with $\epsilon_d=5\times 10^{-3}$, and activated the velocity limiter 
\eqref{tx}--\eqref{tm} with $V_{\max}=15$.
These treatment were not used for the wet bed case.
Due to symmetry of the problem, we perform the computation only on a quarter of the domain with symmetric boundary conditions. We take an unstructured  triangular mesh with
mesh size $\tau_K = 0.5$. The results of two cases are presented in Figure \ref{fig:cd}.
\begin{figure}[ht]
\centering
\includegraphics[width=0.45\textwidth]{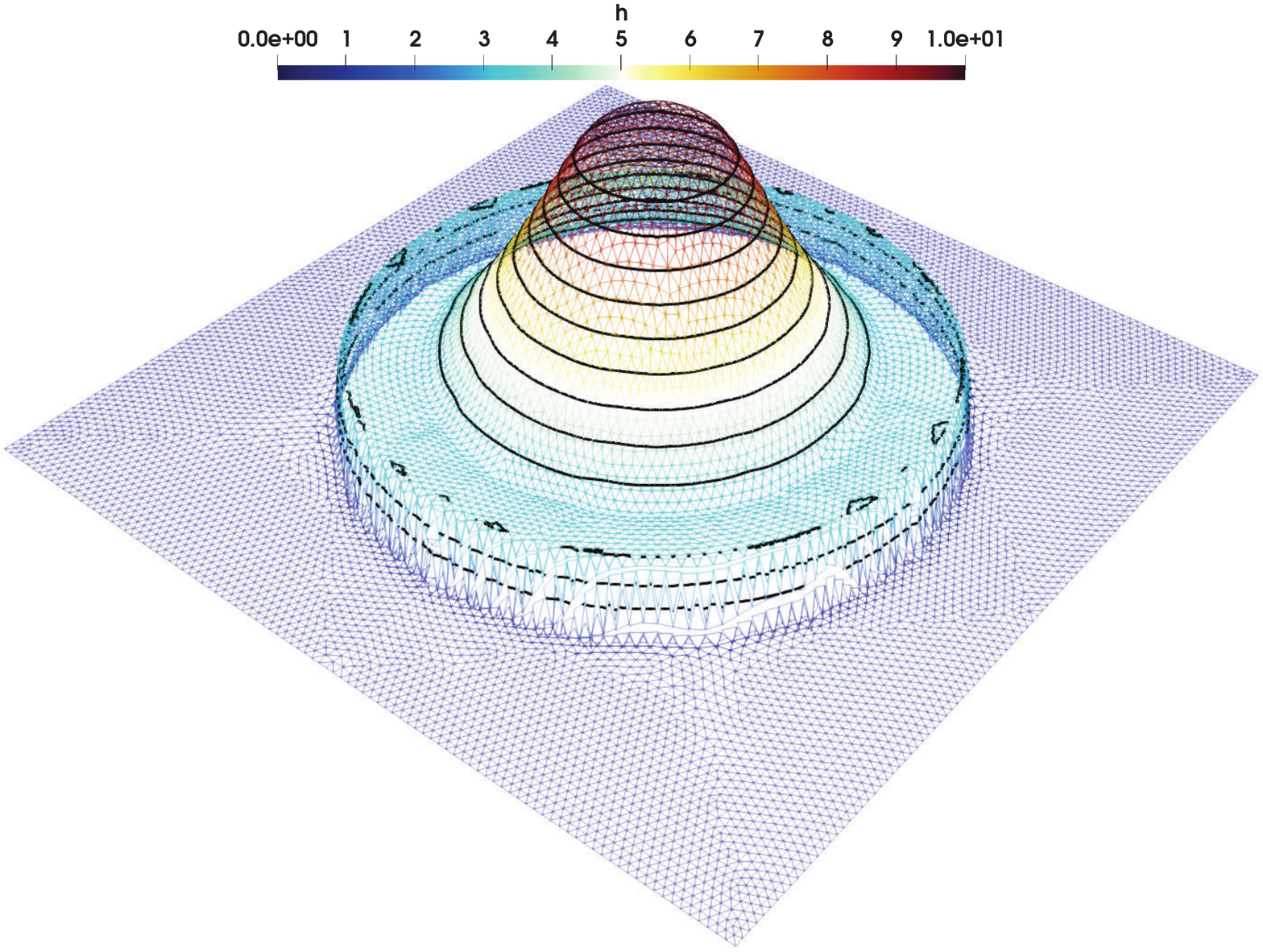}
\includegraphics[width=0.45\textwidth]{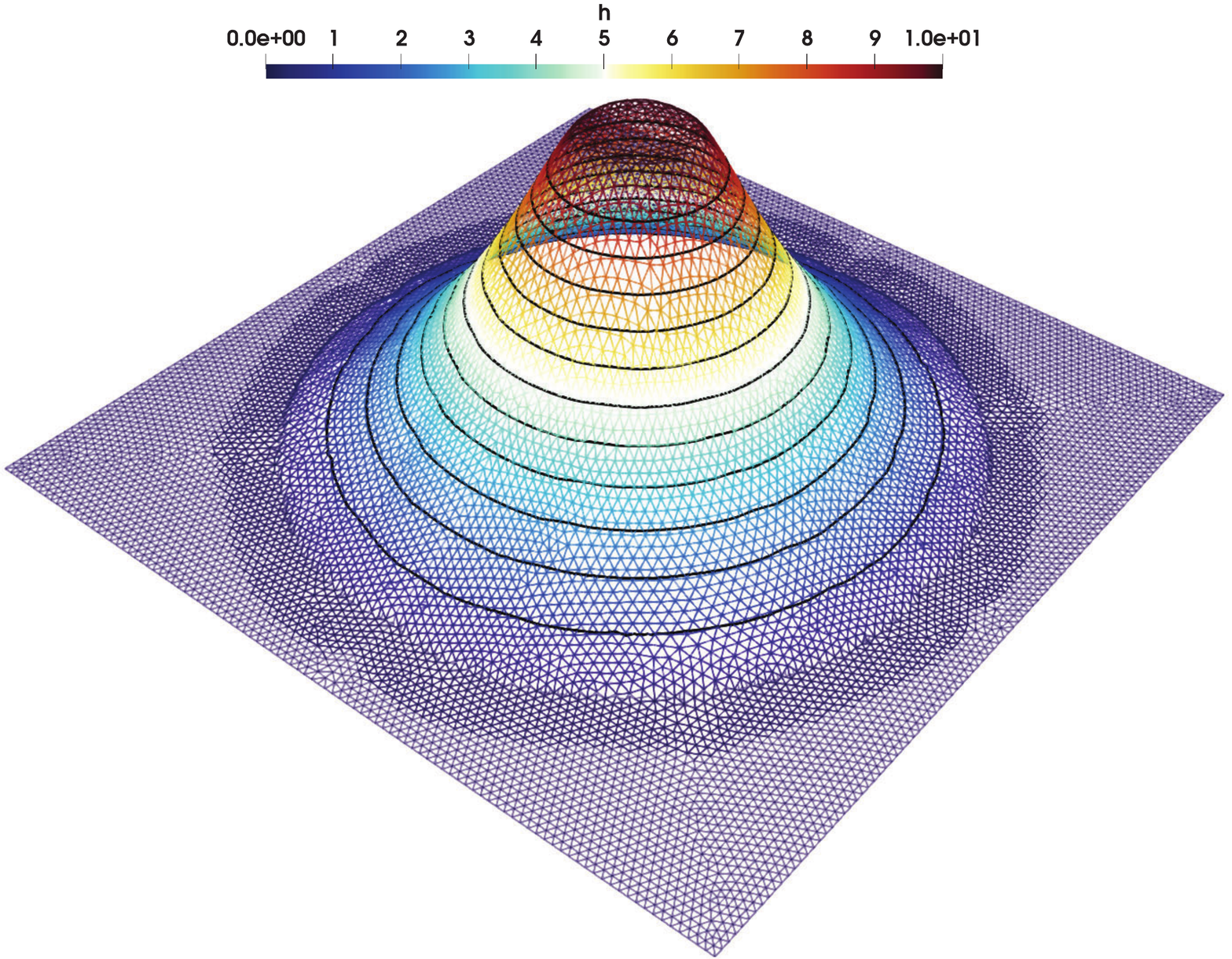}
\caption{Example 4.9.
Contour and surface plots of water height for the circular dam-break problem at $t=0.69$.
Left: web bed. 11 uniform contour lines from $2$ to $9.4$;
Right dry bed. 12 uniform contour lines from $0.01$ to $8.9$.
}
\label{fig:cd}
\end{figure}

\subsection*{Example 4.10: Dam Break on a Closed Channel in 2D}
We consider the problem \cite[Sect. 8.9]{GPC07}
which model dam break on a closed channel.
The domain represents a channel of 75 length and 30 width with
three mounds will wall boundary conditions. 
The shape of the mounds is defined by the function 
$b(x, y) = \max(0, m1, m2, m3)$, where
\begin{align*}
    m_1 =&\; 1- 0.1\sqrt{
(x - 30)^2 + (y-22.5)^2},\\
    m_2 =&\; 1 - 0.1\sqrt{
(x - 30)^2 + (y-7.5)^2},\\
    m_3 =&\; 2.8 - 0.28\sqrt{
(x-47.5)^2 + (y-15)^2}.
\end{align*}
The initial conditions are 
\[
h(x,y,0) = \left\{
\begin{tabular}{ll}
$1.875$, & if $x<16$\\[.7ex]
$10^{-12}$, & otherwise.
\end{tabular}\quad u(x,y,0)=v(x,y,0) = 0
\right.
\]
Due to the moving wet/dry interface, we activate the dry cell limiter with 
$\epsilon_d = 10^{-3}$, and the velocity limiter \eqref{tx}--\eqref{tm} with
$V_{\max}=9$. 
Due to symmetry, we only perform the calculation on half of the domain $[0,75]\times [0,15]$, and apply symmetry boundary conditions on all the boundaries.
Contour plots of the water surface for the simulation results on an unstructured triangular mesh with mesh size $\tau_K = 0.5$ are shown in Figure \ref{fig:xx}
for various times. We observe complex flow structures for this problem, and our scheme 
produces satisfactory results compared with those from \cite[Fig. 15]{GPC07}.

\begin{figure}[ht]
\centering
\includegraphics[width=0.45\textwidth]{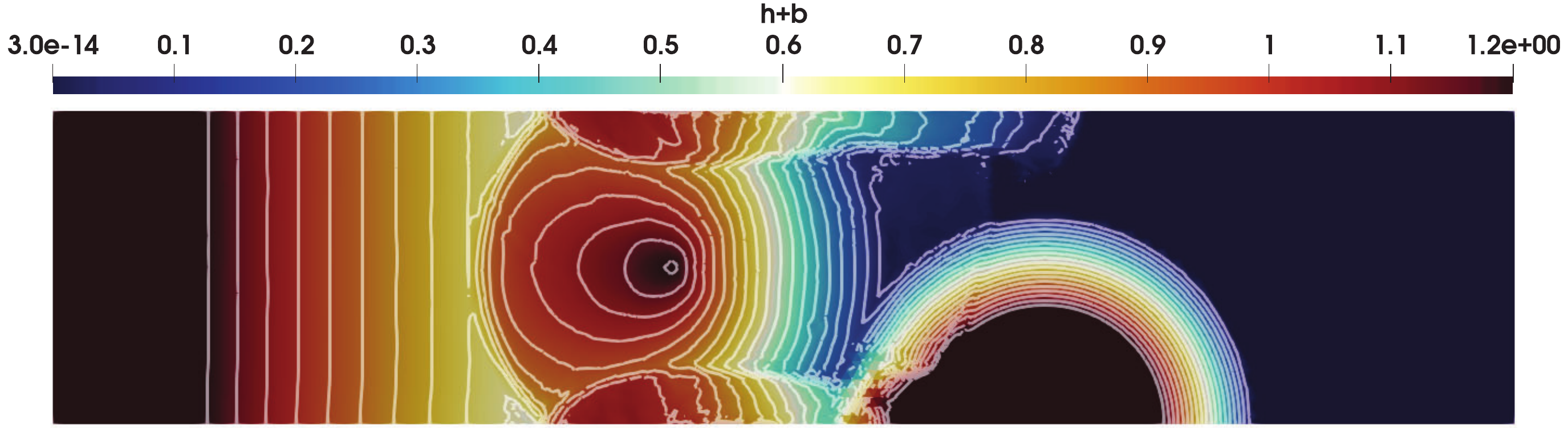}
\includegraphics[width=0.45\textwidth]{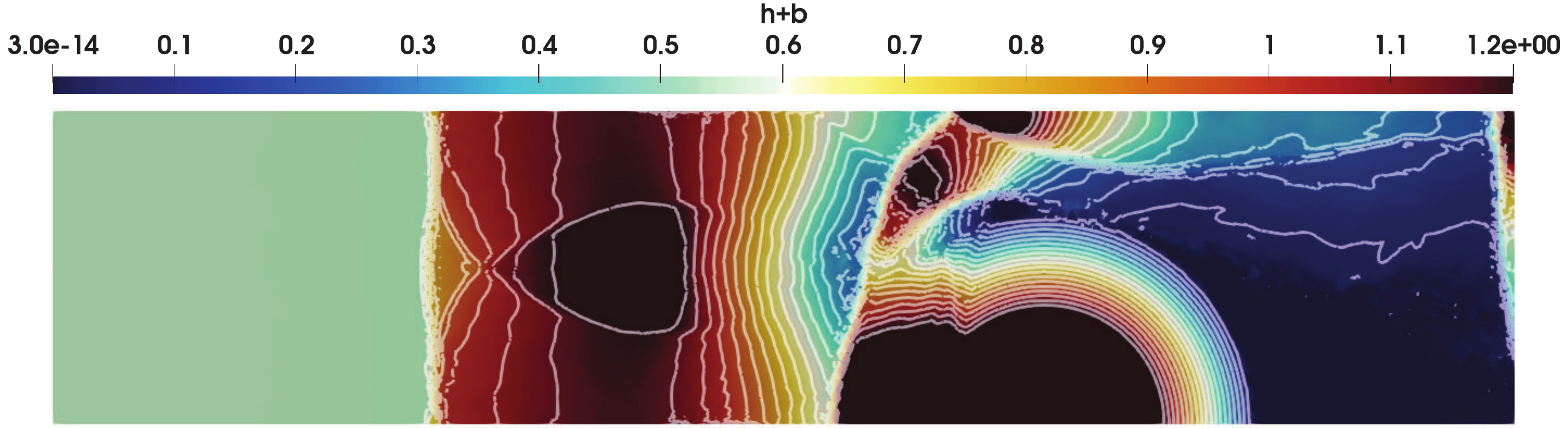}
\includegraphics[width=0.45\textwidth]{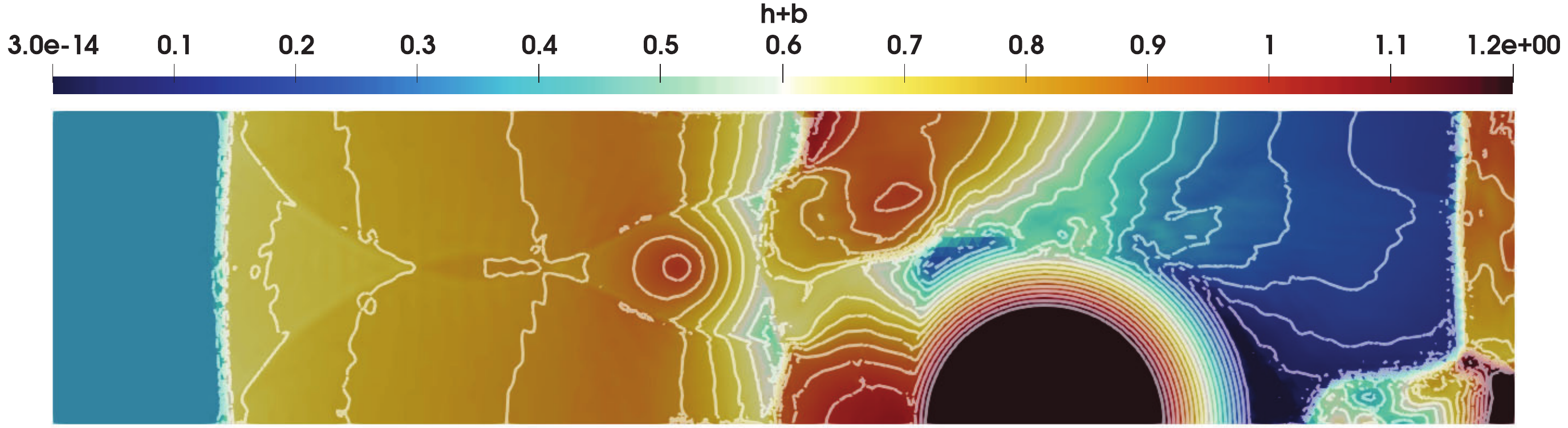}
\includegraphics[width=0.45\textwidth]{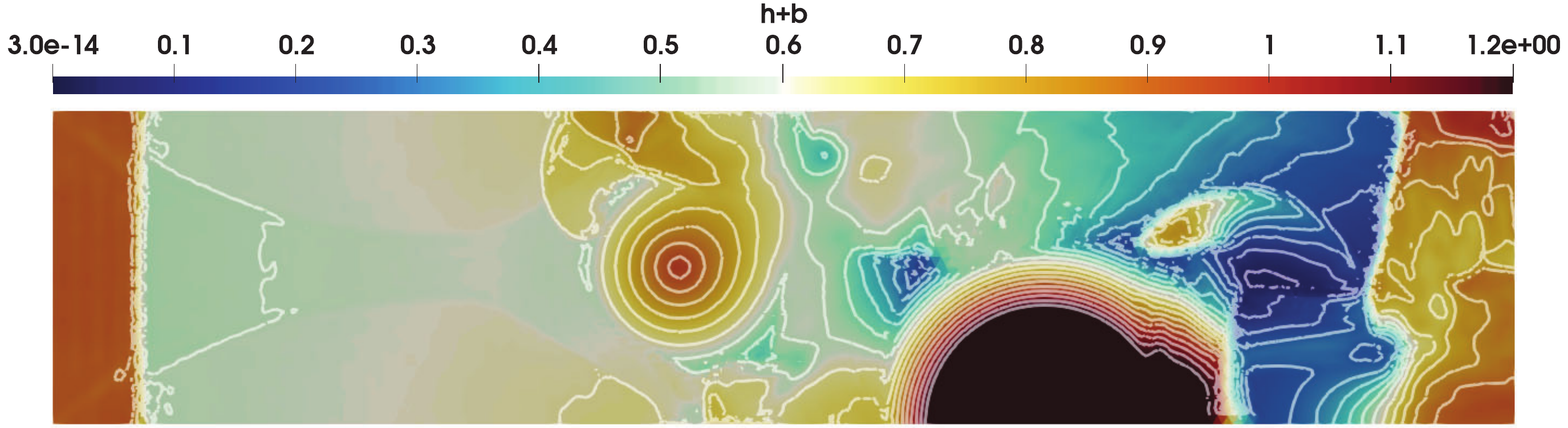}
\includegraphics[width=0.45\textwidth]{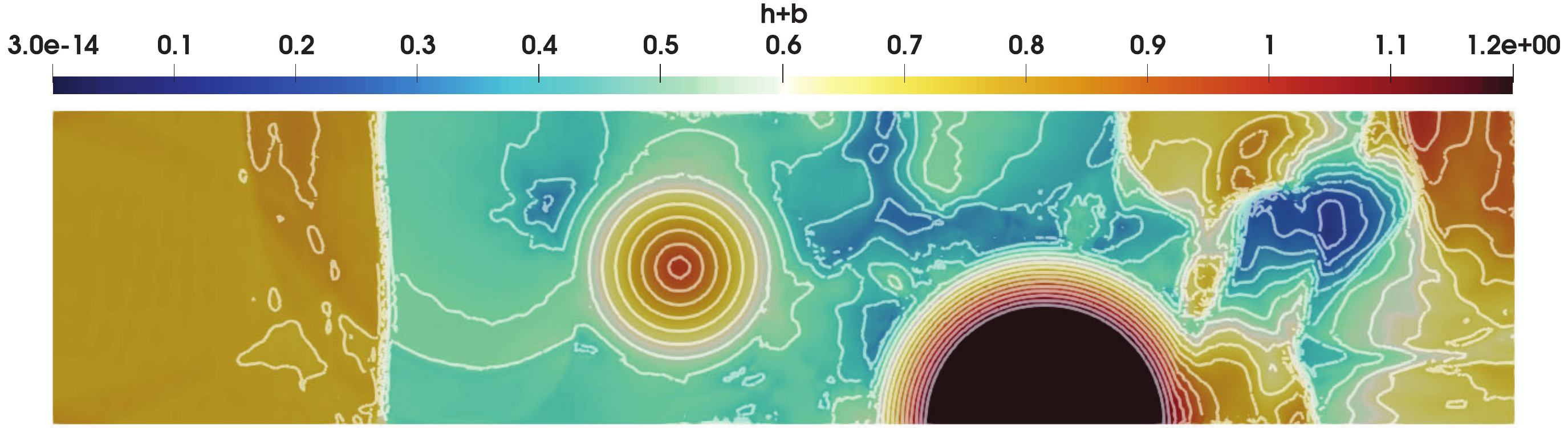}
\includegraphics[width=0.45\textwidth]{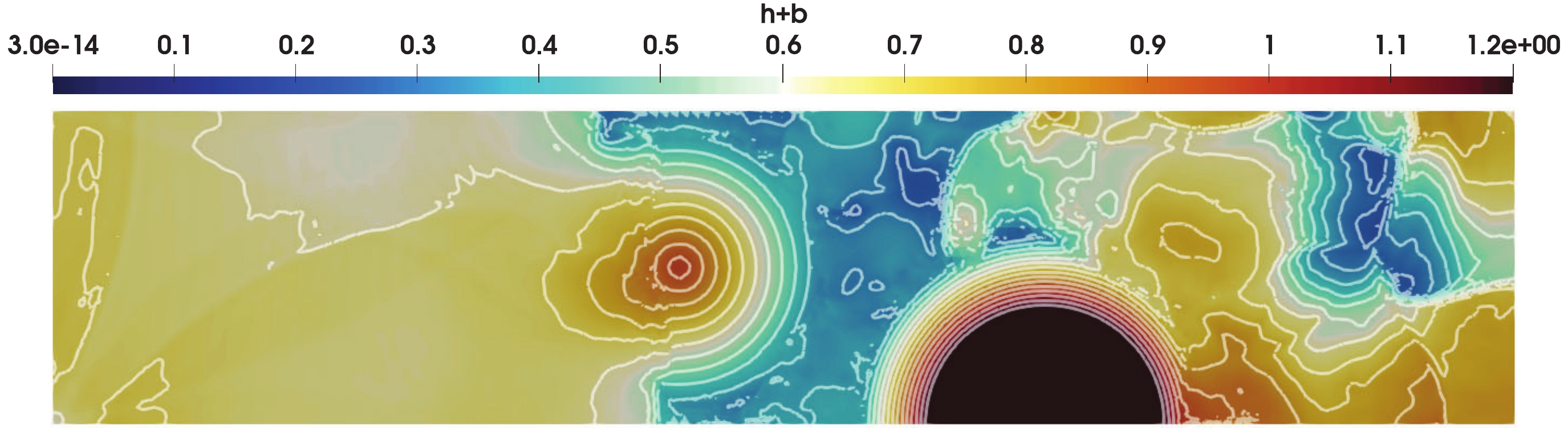}
\includegraphics[width=0.45\textwidth]{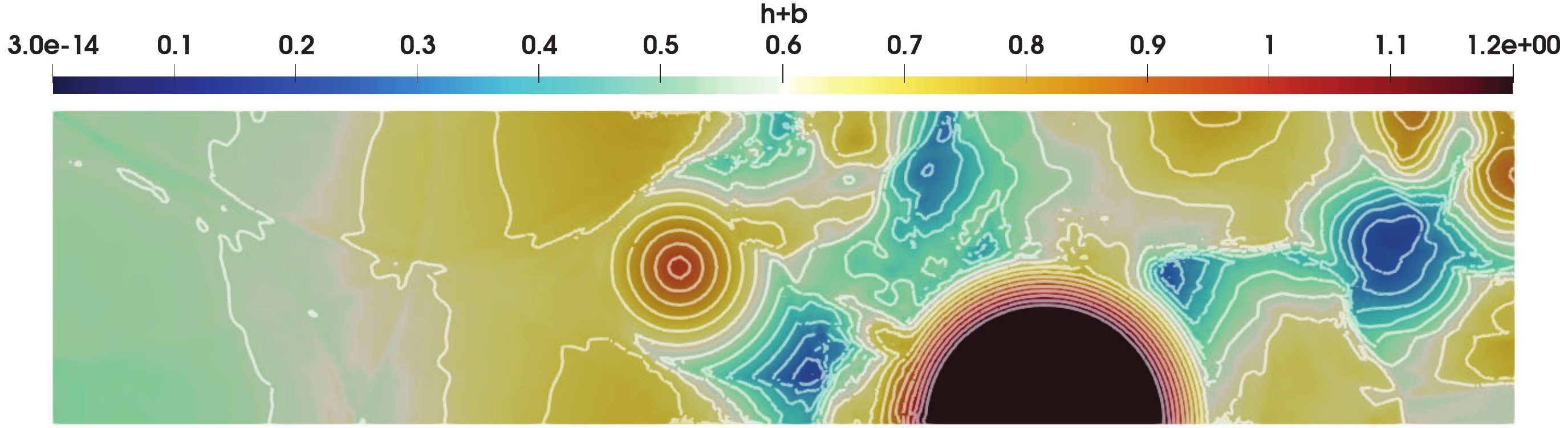}
\includegraphics[width=0.45\textwidth]{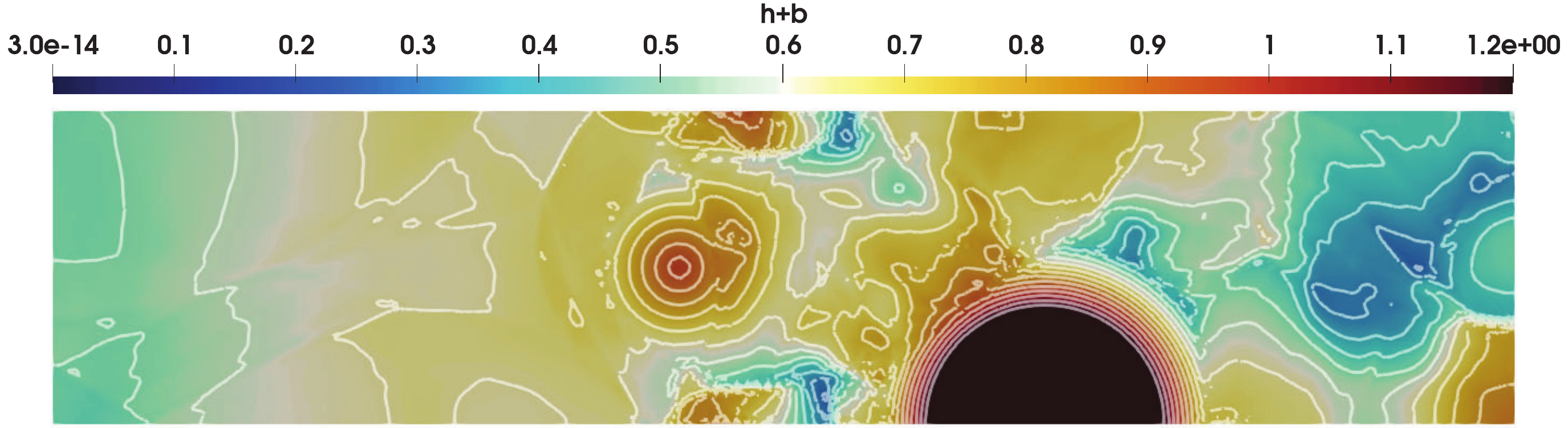}
\caption{Example 4.10.
Contour and surface plots of water surface for the dam-break problem on a closed channel.
20 uniform contour lines from $0$ to $1.2$.
Left to right, top to bottom: $t=5, 10, 15, 20, 25, 30, 35, 40$.
}
\label{fig:xx}
\end{figure}

\section{Conclusion}
We proposed a novel velocity-based DG scheme for the SWEs. Our semidiscrete DG scheme is
locally conservative, entropy stable, and well-balanced. 
We then apply the SSP-RK3 time stepping for the time discretzation, and obtained an explicit locally conservative, well-balanced, and positivity-preserving fully discrete scheme in Algorithm \ref{alg3pt}, where the treatment of strong shocks via a 
characteristic-wise TVB limiter and proper wetting/drying treatment near dry cells was also discussed. Ample numerical examples in 1D and 2D illustrated the good performance of our scheme. Our entropy stable scheme is particularly simple and competitive compared with existing entropy stable DG schemes for SWEs in the literature.

The velocity-based DG scheme can be used to construct robust entropy/energy stable DG schemes for other compressible flow problems, which will be carried out in our future studies.

\

\textbf{Acknowledgement:} The author would like to thank Yulong Xing from Ohio State University for fruitful discussions on the topic.
\bibliographystyle{siam}

\end{document}